\documentclass{birkjour}
\usepackage{amsfonts}
\usepackage{amsthm}
\usepackage{amssymb}
\usepackage{amsmath}
\usepackage{upref}
\usepackage{mathrsfs} 
\usepackage{upref,csquotes}
\usepackage{relsize,exscale}
\usepackage{txfonts}
\usepackage{answers}
\usepackage{graphicx} 
\usepackage{color,tikz}
\def\div{\mathop{\rm div}\nolimits}
\def\uOdt{u_{\mathcal{O},\delta t}}
\def\wOdt{w_{\mathcal{O},\delta t}}
\newtheorem{thm}{Theorem}[section]

\newtheorem{lem}[thm]{Lemma}
\newtheorem{prop}[thm]{Proposition}
\theoremstyle{definition}
\newtheorem{defn}[thm]{Definition}
\theoremstyle{remark}
\newtheorem{rem}[thm]{Remark}

\numberwithin{equation}{section}
\begin{document}
\Newassociation{prooftheorem}{demoprop}{ann}
\renewenvironment{prooftheorem}[1]{\begin{trivlist}\item
\itshape{Proof of Theorem #1}}{\end{trivlist}}
%-------------------------------------------------------------------------
% editorial commands: to be inserted by the editorial office
%
%\firstpage{1} \volume{228} \Copyrightyear{2004} \DOI{003-0001}
%
%
%\seriesextra{Just an add-on}
%\seriesextraline{This is the Concrete Title of this Book\br H.E. R and S.T.C. W, Eds.}
%
% for journalst
%
%\firstpage{1}
%\issuenumber{1}
%\Volumeandyear{1 (2004)}
%\Copyrightyear{2004}
%\DOI{003-xxxx-y}
%\Signet
%\commby{inhouse}
%\submitted{March 14, 2003}
%\received{March 16, 2000}
%\revised{June 1, 2000}
%\accepted{July 22, 2000}
%
%
%
%---------------------------------------------------------------------------
%Insert here the title, affiliations and abstract:
%
\title[Zero-flux boundary condition and its approximations]
{Degenerate parabolic equation with  zero flux boundary condition and its approximations}
\thanks{This work began during a brief stay in LATP (Laboratoire d'Analyse Topologie, Probabilit\'e) at Marseille. The author thanks the members of Laboratory for the warm welcomes. The work on this paper has been supported by the French ANR project CoToCoLa.}
%\TeX-pert.}
%----------Author 1
\author[Gazibo M]{Mohamed Karimou Gazibo}
\address{Laboratoire de Math\'{e}matiques\\
CNRS : UMR 6623- Universit\'{e} de Franche-Comte\\
16, route de Gray\\
25030 Besancon
France}
\email{mgazibok@univ-fcomte.fr}
%----------classification, keywords, date
\date{June 15, 2013}
%----------additions
%\dedicatory{I dedicace this work to my advisor Boris Andreianov}
%%% ----------------------------------------------------------------------
\begin{abstract}
We study a degenerate parabolic-hyperbolic equation with zero flux boundary condition.  The aim of this paper is to prove convergence of numerical approximate solutions towards the unique entropy solution. We propose an implicit finite volume scheme on admissible mesh.% and present some numerical experiments. 
We establish fundamental estimates and prove that the approximate solution converge towards an entropy-process solution. Contrarily to the case of Dirichlet conditions, in zero-flux problem unnatural boundary regularity of the flux is required  to establish that entropy-process solution is the unique entropy solution. In the study of well-posedness of the problem, tools of nonlinear semigroup theory (stationary, mild and integral solutions) were used in [Andreianov, Gazibo, ZAMP, 2013] in order to overcome this difficulty. Indeed, in some situations including the one-dimensional setting, solutions of the stationary problem enjoy additional boundary regularity. Here, similar arguments are developed based on the new notion of integral-process solution that we introduce for this purpose.
\end{abstract}
%%% ----------------------------------------------------------------------
\maketitle
%%% ----------------------------------------------------------------------
%\tab\elleofcontents
%%%%%%%%%%%%%%%%%%%%%%%%%%%%%%%%%%%%%%%%%%%%%%%%%%%%%%%%%%%%%%%%%%%%%%%%%%%%%%%%%%%%%%%%%%%%%%%%%%%%%%%%%%%%%%%%%%%%%%%%%%%%%%%%%%%%%%%%%%%%%%%%%%%%%%%%%%%%%%%%%
\section{Introduction}
Let $\Omega$ be a bounded open set of $\mathbb R^\ell$, $\ell\geq1$, with a Lipschitz boundary $\partial\Omega$ and $\eta$\\ the unit normal to $\partial\Omega$ outward to $\Omega$. We consider the zero-flux boundary problem 
\begin{equation*}
(P)\left\{\begin{array}{rll}
u_t+\div f(u)-\Delta\phi(u)&=0         &\mbox{ in } \;\;\; Q=(0,T)\times\Omega,\\
                     u(0,x)&=u_0(x)   &\mbox{ in } \;\;\; \Omega,\\
(f(u)-\nabla\phi(u)).\eta&=0           &\mbox{ on } \;\;\; \Sigma=(0,T)\times\partial\Omega.
\end{array}\right.
\end{equation*}
The function $f$ is continuous and satisfy:
\begin{equation}\label{f}\tag{H1}
f(0)=f(u_{\max})=0 \mbox{ for some  } u_{\max}>0.
\end{equation}
We suppose that the initial data $u_0$ takes values in $[0,u_{\max}]$. In this case $[0,u_{\max}]$ will be an invariant domain for the solution of $(P)$ (see \cite{BG}). The function $\phi$ is  non decreasing Lipschitz continuous in $[0,u_{\max}]$. Formally $\Delta(\phi(u))=\div(\phi'(u)\nabla u)$. Then, if $\phi'(u)=0$ for some $(t,x)\in Q$, the diffusion term vanishes so that $(P)$ is a degenerate parabolic-hyperbolic problem. In our context, we suppose as in \cite{BG} that there exists a real value $u_c$  with $0\leq u_c\leq u_{\max}$ such that  for $u\leq u_c$, the problem $(P)$ is hyperbolic. This means that $\phi\equiv0$ on $[0,u_c]$ and $\phi$ is strictly increasing in $[u_c,u_{\max}]$. Also as in \cite{BG}, we assume that the couple $(f,\phi)$ is non-degenerate, this means that for all $\xi\in\mathbb R^{\ell}$, $\xi\neq 0$, the functions $\lambda\longmapsto\displaystyle\sum_{i=1}^\ell\xi_if_i(\lambda)$ is not affine on the non-degenerate sub intervals of $[0, u_c]$. It is well know that uniqueness of weak solution of degenerate hyperbolic-parabolic problem is not ensured, and one has to define a notion of entropy solution in the sense of Carrillo \cite{CAR} (see in the strictly hyperbolic case Kruzhkov \cite{KRU}) to recover uniqueness. 
Inspired by \cite{BFK1}, we defined in \cite{BG}, a suitable notion of entropy solution for $(P)$. A measurable function $u$ taking values on $[0,u_{\max}]$ is called an entropy solution of the initial-boundary value  problem $(P)$ if $\phi(u)\in L^2(0,T;H^1(\Omega))$ and $\forall k\in  [0,u_{\max}]$,  $\forall\xi\in \mathcal{C}^\infty([0,T)\times\mathbb R^\ell)^+$, the following inequality   hold 
\begin{align}\label{ESzeroflux}
&\displaystyle\int_0^T\!\!\!\int_\Omega\displaystyle \left\{|u-k|\xi_t+sign(u-k)\Bigl[f(u)-f(k)-\nabla\phi(u)\Bigr].\nabla\xi\displaystyle\right\}dxdt\nonumber\\&+\displaystyle\int_0^T\!\!\!\int_{\partial\Omega} \left|f(k).\eta(x)\right|\xi(t,x) d{\mathcal{H}}^{\ell-1}(x)dt+\displaystyle\int_\Omega |u_0-k|\xi(0,x)dx\geq 0.
\end{align}
Let us recall the main theoretical results on problem $(P)$ obtained in \cite{BG}.  We prove existence of solution satisfying \eqref{ESzeroflux}, for any space dimension in the case $0<u_c<u_{\max}$. Uniqueness is obtained for one space dimension. Remark that uniqueness is also true in multi-dimensional situation in two extreme cases: $u_c=0$ (non-degenerate parabolic case, see \cite{BF}) and $u_c=u_{\max}$ (pure hyperbolic case, (see \cite{BFK1})). We refer to Appendix 2 for some explanations.\\
In this paper, we choose an implicit finite volume scheme for the discretization of the parabolic equation $(P)$. Under suitable assumptions on the numerical fluxes, it is shown that the considered schemes are $L^\infty$ stable and the discrete solutions satisfy  some weak BV inequality and $H^1$ estimates.  We prove also  space and time translation estimates on the diffusion fluxes, which are the keys to the proof of convergence of the scheme. We prove existence of discrete solution by using Leray-Schauder topological degree. The approximate solutions are shown to satisfy the appropriate discrete entropy inequalities. Using the weak BV and $H^1$ estimates, the approximate solutions are also shown to satisfy  continuous entropy inequalities. It remains to prove that the sequence of approximate solutions satisfying this continuous entropy inequalities converge towards an unique entropy solution.  In \cite{VOV},  Michel and Vovelle use the concept of 'entropy-process solution' introduced by  Gallou\"{e}t and al (see e.g. \cite{EGHMichel, VOV, Claire}) for Dirichlet boundary problem which is similar to the notion of measure valued solutions of Diperna \cite{Diperna}. They proved that approximate solutions converge towards an entropy-process solution as the mesh size tends to zero. Using doubling of variables method, they showed that the entropy-process solution is unique and is also a entropy solution of Dirichlet problem. In the case of zero flux boundary condition, some difficulty due to lack of regularity for the boundary flux appears (see \cite{BG}). We are not able to obtain uniqueness by the  doubling of variables method. Thus, the only notion of entropy-process solution is not enough to prove convergence towards the entropy solution. To solve this difficulty,  we found it useful to consider the general evolution problem of the form: 
\begin{align*}
(E)\left\{\begin{array}{rll}
v'(t)+A(v(t))&=0         &\mbox{ on } \;\;\; (0,T);\\
                    v(0)&=u_0.
\end{array}\right.
\end{align*}
 We  propose a new notion of solution called integral-process solution for the abstract evolution problem $(E)$. This notion is presented in detail in the appendix $1$. We prove that this new notion of integral-process solution coincides with the unique integral solution. Then, we apply this notion to the problem $(P_1)$ and prove that the approximate solutions converge to an integral-process solution.  Then we conclude  that it is an entropy solution.\\% Notice that, in the context of  pure hyperbolic problem, this notion of integral-process solution is not necessary. The strong trace of the boundary flux is proved by regularity of $f$ and the result of Vasseur (see also \cite{BFK1}). This permit to prove as in the case of Dirichlet problem that entropy process solution is entropy solution.\\
 
The rest of this paper is organized as follows. In Section 2, we present our implicit scheme. In Section 3, we prove a priori estimates, the discrete entropy inequalities and existence of discrete solution in Section 4. We propose in Section 5 a continuous entropy inequality, and the convergence result follows in Section 6. %In section 7, we deal with numerical experimentation.
 Finally, in Appendices we study the abstract evolution equation $(E)$ and prove uniqueness of entropy solution in one space dimension for degenerate parabolic equation.
%%%%%%%%%%%%%%%%%%%%%%%%%%%%%%%%%%%%%%%%%%%%%%%%%%%%%%%%%%%%%%%%%%%%%%%%%%%%%%%%%%%%%%%%%%%%%%%%%%%%%%%%%%%%%%%%%%%%%%%%%%%%%%%%%%%%%%%%%%%%%%%%%%%%%%%%%%%%%%%%%%%%%
\section{Presentation of a finite volume scheme for degenerate parabolic problem with zero flux boundary condition}
\begin{figure}[bht]
\begin{center} 
\scalebox{1.2}% pour redimensionner
{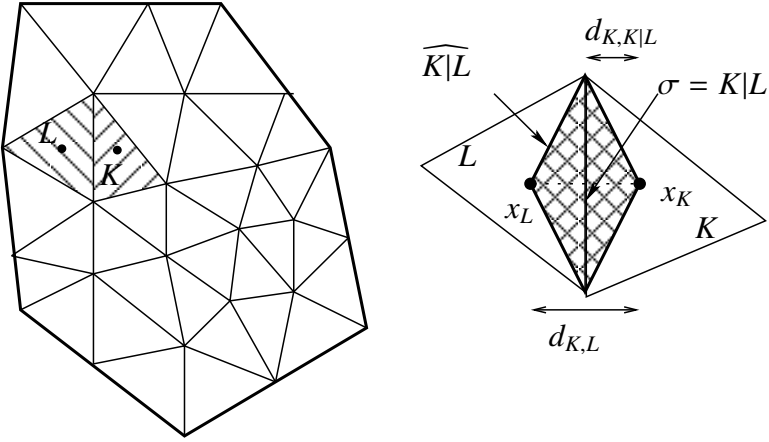}
\end{center}
\caption{Control volumes, centre, diamonds}
\label{fig:volumefini1}
\end{figure}
In this section, we consider the problem $(P_1)$ and construct a monotone finite volume scheme to approximate the solution. Let $\delta t>0$ be  the time step. Let $\mathcal{O}$ be a family of disjoint connected polygonal subsets called control volumes of $\Omega$ such that $\overline{\Omega}$ is the union of the closures of the elements of this family and such that the common interface of two control volumes is included in the hyperplane of $\mathbb R^\ell$. Let $h$ be the upper bound for maximum size of the mesh: $h=\sup\{\mbox{Diam}(K), K\in\mathcal{O}\}$. We suppose that there exists $\alpha>0$ such that:
\begin{align}\label{reg1}
\alpha h^\ell\leq m(K),\;\; m(\partial K)\leq\frac{1}{\alpha}h^{\ell-1},\;\;\;\;\forall K\in\mathcal{O},
\end{align}
then the estimate on  the number $|\mathcal{O}|$ of control volumes is
\begin{align}\label{reg2}
|\mathcal{O}|\leq\frac{m(\Omega)}{\alpha}h^{-\ell},
\end{align}
where $m(K)$ is the $\ell-$ dimensional Lebesgue measure of $K$ and $m(\partial K)$ is the $(\ell-1)-$ dimensional Lebesgue measure of $\partial K$. If $K$ and $L$ are two control volumes having an edge $\sigma$ in common, we say that $L$ is  a neighbor of $K$  and we write   $L\in\mathcal{N}(K)$. We  sometimes denote by $K|L$ the common edge $\sigma$ between $K$ and $L$ and by $n_{K,\sigma}$ the unit normal to $\sigma$, oriented from $K$ to $L$. Moreover, $\bar{\mathlarger\varepsilon}_K$ denotes the set of all edges for any control volumes $K$. If $K$ has at least one common edge with boundary $\partial\Omega$, we denote by $\mathlarger\varepsilon_K^{ext}$ the set of these boundaries edges, that can be regarded as $\mathlarger\varepsilon_K^{ext}=\{\sigma\in\bar{\mathlarger\varepsilon}_K, m(\sigma\cap\partial\Omega)>0\}$. Eventually, if the control volume $K$ has no common edges with a part of  boundary $\partial\Omega$ then $\mathlarger\varepsilon_K^{ext}=\emptyset$.  In all case, for all control volume $K\in\mathcal{O}$, we have $\mathlarger\varepsilon_K=\bar{\mathlarger\varepsilon}_K\backslash\mathlarger\varepsilon_K^{ext}$. Because we consider the zero-flux boundary condition, we don't need to distinguish between interior and exterior control volumes, only inner interfaces between volumes are needed in order to formulate the scheme.
We consider here the admissible mesh of $\Omega$ ( see for e.g. \cite{VOV}), we mean that there exists a family of points $(x_K)_{ K\in\mathcal{O}}$ such that the straight line $\overline{x_Kx_L}$ is orthogonal to the interface $K|L$. We denote by $d_{K,L}=|x_K-x_L|$ the distance between $x_K$ and $x_L$ and by $d_{K,\sigma}$ the distance between $x_K$ and the interface $\sigma$ (see  Figure \ref{fig:volumefini1}). The point $x_K$ is referred as the centre of $K$.  To simplify the analysis, we consider that $x_K\in K$ (in general, this assumption can be relaxed, e.g., one can consider so called Delaunay simplicial meshes). We denote by $\tau_{K,\sigma}$ the 'transmissibility' through $\sigma$ defined by $\tau_{K,\sigma}=\frac{m(\sigma)}{d_{K,\sigma}}$ if $\sigma\in\mathlarger\varepsilon_K^{ext}$, in addition we denote $\tau_{K|L}=\frac{m(K|L)}{d_{K,L}}$. The diamond denoted by $\widehat{K|L}$ is a  convex hull constructed from neighbor centers $x_K$, $x_L$ and $K|L$. The diamonds are disjoint and cover $\Omega$ up to an h-neighborhood of $\partial\Omega$.  Notice that the $\ell-$ dimensional measure $m(\widehat{K|L})$ of $\widehat{K|L}$ equals to  $\frac{d_{K,L}}{\ell}m(K|L)$  (see  Figure \ref{fig:volumefini1}). 

 A discrete function $w$ on the mesh $\mathcal{O}$ is a set $(w_K)_{K\in\mathcal{O}}$. If $w_K$, $v_K$ are discrete functions, the corresponding $L^2(\Omega)$ scalar product and norm can be computed as
\begin{align*}
(w_{\mathcal{O}},v_{\mathcal{O}})_{L^2(\Omega)}=\displaystyle\sum_{K\in\mathcal{O}}m(K)w_Kv_K;\;\;\;||w_{\mathcal{O}}||_{L^2(\Omega)}^2=\displaystyle\sum_{K\in\mathcal{O}}m(K)|w_K|^2.
\end{align*}
In addition, we can define the positive (but not definite) product and the corresponding "discrete $H^1_0$ semi-norm" by
\begin{align}\label{H1seminorm}
(w_{\mathcal{O}},v_{\mathcal{O}})_{H_\mathcal{O}}=\ell\displaystyle\sum_{K\in\mathcal{O}}\sum_{L\in\mathcal{N}(K)}\frac{m(K|L)}{d_{K,L}}(w_L-w_K)(v_L-v_K);\;\;\;|w_{\mathcal{O}}|_{H_\mathcal{O}}^2=\biggl((w_{\mathcal{O}},w_{\mathcal{O}})_{H_\mathcal{O}}\biggr).
\end{align}
We define the discrete gradient $\nabla_{\mathcal{O}} w_{\mathcal{O}}$ of a constant per control volume function $w_{\mathcal{O}}$ as the constant per diamond $\widehat{K|L}$, $\mathbb R^\ell$-valued function with values
 \begin{align}\label{gradiantdiscret}
(\nabla_{\mathcal{O}} w_{\mathcal{O}})_{\widehat{K|L}}=\nabla_{\widehat{K|L}} w_{\mathcal{O}}:=\ell\frac{w_L-w_K}{d_{K,L}}\eta_{K,L}.
\end{align} 
 For the approximation of the convective term, we consider the numerical convection fluxes $F_{K,\sigma}:\mathbb R^2\longrightarrow\mathbb R$  for $K\in\mathcal{O},\sigma\in\bar{\mathlarger\varepsilon}_K$ \\%(where $F_{K,L}$ if $\sigma=K|L$ ). \\
The numerical convection fluxes are monotone:
\begin{align}\label{monotony}
&F_{K,\sigma}:[0,u_{\max}]^2\longrightarrow\mathbb R; (a,b)\longmapsto F_{K,\sigma}(a,b) \nonumber\\&\mbox{ is nondecreasing with respect to   }   a \mbox{ and nonincreasing with respect to } b. 
\end{align}
The numerical convection fluxes are conservative:
\begin{align}\label{conservativity} 
\mbox{For all } \sigma=K|L, \mbox{ for all } a,b\in[0,u_{\max}];  F_{K,L}(a,b)=-F_{L,K}(b,a). 
\end{align}
The numerical convection fluxes are regular:
\begin{align}\label{regularity}
 &F_{K,\sigma} \mbox{ is Lipschitz continuous }\mbox{ and admits }  m(\sigma)M  \mbox{ as Lipschitz constant on }  \nonumber\\&[0,u_{\max}]. 
\end{align}
The numerical convection fluxes are consistent:
\begin{align}\label{consistency}
\mbox{ For all } s\in[0,u_{\max}], F_{K,\sigma}(s,s)=m(\sigma)f(s).n_{K,\sigma}.
\end{align}

The Godunov, the splitting flux of Osher and Rusanov schemes may be the most common examples of schemes with fluxes satisfying \eqref{monotony}-\eqref{consistency}.\\
Notice that the hypothesis \eqref{regularity} and \eqref{consistency} entail the bound
\begin{align}\label{borne}
\forall a,b\in[0,u_{\max}],\;\;|F_{K,\sigma}(s,s)|\leq(||f||_{L^\infty}+M u_{\max})m(\sigma)
\end{align}
The discrete unknowns $u_K^{n+1}$ for all control volume $K\in\mathcal O$, and $n\in\mathbb N$ are defined  thanks to the following relations:  first we initialize the scheme by
\begin{align}\label{esti0}
u_K^0=\frac{1}{m(K)}\displaystyle\int_K u_0(x)dx\;\;\;\forall K\in\mathcal O,
\end{align}
then,  we use the implicit scheme for the discretization of problem $(P)$:\\ $\forall n>0,\;\forall K\in\mathcal O$,
\begin{align}\label{esti1}
 m(K)\frac{u_K^{n+1}-u_K^n}{\delta t}+\displaystyle\sum_{\sigma\in\mathlarger\varepsilon_K}F_{K,\sigma}(u_{K}^{n+1}\!,u_{K,\sigma}^{n+1})-\displaystyle\sum_{\sigma\in\mathlarger\varepsilon_K}\tau_{K,\sigma}\biggl(\phi(u_{K,\sigma}^{n+1})-\phi(u_{K}^{n+1})\biggr)=0.
\end{align}
If the scheme has a solution, we will say that the piecewise constant function $\uOdt(t,x)$ defined by: 
\begin{align}
\uOdt(t,x)=u^{n+1}_K \mbox{ for } x\in K \mbox{ and } t\in]n\delta t,(n+1)\delta t], \mbox{ a.e.} 
\end{align}
is an approximate solution to $(P)$. 
\begin{rem}\label{div}
\begin{enumerate}
\item  Notice that using relation \eqref{consistency}  and the fact that for all $s\in \mathbb R$  $\div_xf(s)=0$, one gets
\begin{align}
\forall s\in[0,u_{\max}], \forall K\in\mathcal{O}, \displaystyle\sum_{\sigma\in\bar{\mathlarger\varepsilon}_K}F_{K,\sigma}(s,s)=0.
\end{align}
This is equivalent to:
\begin{align}\label{Kext}
\forall s\in[0,u_{\max}], \forall K\in\mathcal{O}, \displaystyle\sum_{\sigma\in\mathlarger\varepsilon_K} F_{K,\sigma}(s,s)+\displaystyle\sum_{\sigma\in\mathlarger\varepsilon_K^{ext}} F_{K,\sigma}(s,s)=0.
\end{align}
\item Notice that the prescribed zero flux boundary condition is in fact included in \eqref{esti1}. One can extend the summation over $\sigma\in\bar{\mathlarger\varepsilon}_K$, and by convention regard the fluxes as:
\begin{equation}\label{neumann}
F_{K,\sigma}(u_{K}^{n+1}\!,u_{K,\sigma}^{n+1})=\left\{\begin{array}{lrc}
F_{K,\sigma}(u_{K}^{n+1}\!,u_{L}^{n+1})        &\mbox{ if } \;\;\; \sigma \in K|L,\\
0  &\mbox{ if } \;\;\; \sigma\in\mathlarger\varepsilon_K^{ext}.
\end{array}\right.
\end{equation}
\begin{equation}\label{neumanndiffusion}
\tau_{K,\sigma}\biggl(\phi(u_{K,\sigma}^{n+1})-\phi(u_{K}^{n+1}\biggr)=\left\{\begin{array}{lrc}
\tau_{K|L}\biggl(\phi(u_{L}^{n+1})-\phi(u_{K}^{n+1})\biggr)       &\mbox{ if } \;\;\; \sigma \in K|L,\\
0  &\mbox{ if } \;\;\; \sigma\in\mathlarger\varepsilon_K^{ext}.
\end{array}\right.
\end{equation} 
%\item  We expect that also that with explicit scheme all the results of this paper can be  obtain, under a convenient CFL condition on time step and mesh size.
%But we are not able to prove that $\uOdt$ takes values in $[0,u_{ \max}]$ for all $K\in\mathcal{O}$ and $n>0$.
\end{enumerate}
\end{rem}
%%%%%%%%%%%%%%%%%%%%%%%%%%%%%%%%%%%%%%%%%%%%%%%%%%%%%%%%%%%%%%%%%%%%%%%%%%%%%%%%%%%%%%%%%%%%%%%%%%%%%%%%%%%%%%%%%%%%%%%%%%%%%%%%%%%%%%%%%%%%%%%%%%%%%%%%%%%%%%%%%%%%
\section{Discrete entropy inequalities}
This part is devoted to  discrete entropy inequalities. We recall some notations (\cite{EGHMichel}): \\ Denote by $a\bot b=\min(a,b)$ and $a\top b=\max(a,b)$. We define $\eta_k^+(s)=(s-k)^+=s\top k-k$, (respectively $\eta_k^-(s)=(s-k)^-=s\bot k-k$) and  the associated fluxes-functions $\Phi^{\pm}_k$ called entropy fluxes 
\begin{align*}
&\Phi_k^+(s)=sign^+(s-k)(f(s)-f(k))=f(s\top k)-f( k);\\
&\Phi_k^-(s)=sign^-(s-k)(f(s)-f(k))=f(s\bot k)-f(k);\\
&\Phi_k(s)=sign(s-k)(f(s)-f(k)).
\end{align*}
Therefore, the numerical sub and super entropy fluxes functions are defined by the formulas
\begin{align*}
&\Phi_{K,\sigma,k}^{+}(a,b)=F_{K,\sigma}(a\top k,b\top k)-F_{K,\sigma}(k,k);\\
&\Phi_{K,\sigma,k}^{-}(a,b)=F_{K,\sigma}(k,k)-F_{K,\sigma}(a\bot k,b\bot k);\\
&\Phi_{K,\sigma,k}(a,b)=F_{K,\sigma}(a\top k,b\top k)-F_{K,\sigma}(a\bot k,b\bot k).
\end{align*}
From now, we have the following the discrete entropy inequalities.
\begin{lem}
Assume that \eqref{reg2}, \eqref{monotony}- \eqref{consistency}  hold. Let $u_{\mathcal{O},\delta t}$ be an approximate solution of the problem $(P)$ defined by \eqref{esti0}, \eqref{esti1}, . Then for all $k\in[0,u_{\max}]$, for all $K\in\mathcal{O}$, $n\geq0$ the following discrete sub-entropy inequalities hold: 
\begin{align}\label{entropydiscrete}
&\frac{\eta^+_k(u^{n+1}_K)-\eta^+_k(u^{n}_K)}{\delta t}m(K)+\displaystyle\sum_{\sigma\in\mathlarger\varepsilon_K}\Phi_{K,\sigma,k}^{+}(u_K^{n+1},u_{K,\sigma}^{n+1})\nonumber\\&
-\displaystyle\sum_{K|L}\tau_{K|L}\biggl(\eta^+_{\phi(k)}(\phi(u_{L}^{n+1}))-\eta^+_{\phi(k)}(\phi(u_{K}^{n+1}))\biggr)\nonumber\\&\leq\displaystyle\sum_{\sigma\in\mathlarger\varepsilon_K^{ext}}sign^+(u_K^{n+1}-k)m(\sigma)f(k)n_{K,\sigma}.
\end{align}
%where 
%\begin{equation}\label{neumann0}
%\Phi_{K,\sigma,k}^{+}(u_K^{n+1},u_{K,\sigma}^{n+1})=\left\{\begin{array}{rll}
%\Phi_{K,\sigma,k}^{+}(u_K^{n+1},u_{L}^{n+1})       &\mbox{ if } \;\;\; \sigma \in K|L,\\
%-sign^+(u^{n+1}_K-k)F_{K,\sigma}(k,k)  &\mbox{ if } \;\;\; \sigma\in\mathlarger\varepsilon_K^{ext}.
%\end{array}\right.
%\end{equation}
Also the discrete super-entropy inequalities are satisfied (i.e., $\eta^+_k$, $\Phi_{K,\sigma,k}^{+}$, $sign^+$ can be replaced by $\eta^-_k$, and $\Phi_{K,\sigma,k}^{-}$, $sign^-$ in  \eqref{entropydiscrete}. 
\end{lem} 
Notice that, if  for all $K\in\mathcal{O}$, $u_K^{n+1}$ satisfy both discrete sub-entropy inequality and discrete super-entropy inequality, then $u_K^{n+1}$ can be seen as  a discrete entropy solution in $K\times]n\delta t,(n+1)\delta t]$.
\begin{proof}
%Thanks to the Remark \ref{div} and relation \eqref{neumann}, \eqref{neumanndiffusion} the equation \eqref{esti1} can be rewritten in the following form:
%\begin{align}\label{esti4} 
%m(K)\frac{u_K^{n+1}-u_K^n}{\delta t}+\displaystyle\sum_{\sigma\in\mathlarger\varepsilon_K}F_{K,\sigma}(u_{K}^{n+1}\!,u_{K,\sigma}^{n+1})-\displaystyle\sum_{L\in\mathcal{ N}(K)}\tau_{K|L}\biggl(\phi(u_{K,\sigma}^{n+1})-\phi(u_{K}^{n+1})\biggr)=0;.
%\end{align}
Thanks to the Remark \ref{div}, the constant $k\in[0,u_{\max}]$ is solution of:
\begin{align}\label{esti5}
m(K)\frac{k-k}{\delta t}+\displaystyle\sum_{\sigma\in\mathlarger\varepsilon_K}F_{K,\sigma}(k,k)-\displaystyle\sum_{K|L}\tau_{K|L}\biggl(\phi(k)-\phi(k)\biggr)=-\displaystyle\sum_{\sigma\in\mathlarger\varepsilon_K^{ext}}F_{K,\sigma}(k,k).
\end{align}
Substracting from the equality \eqref{esti1} the equality \eqref{esti5}, we obtain:
\begin{align}\label{esti6}
&\frac{1}{\delta t}\biggl((u_K^{n+1}-k)-(u_K^n-k)\biggr) m(K)+\displaystyle\sum_{\sigma\in\mathlarger\varepsilon_K}\biggl(F_{K,\sigma}(u_{K}^{n+1}\!,u_{K,\sigma}^{n+1})-F_{K,\sigma}(k,k)\biggr)\nonumber\\&-\displaystyle\sum_{L\in\mathcal{ N}(K)}\tau_{K|L}\biggl[\biggl(\phi(u_{L}^{n+1})-\phi(k)\biggr)-\biggl(\phi(u_{K}^{n+1})-\phi(k)\biggr)\biggr]=\displaystyle\sum_{\sigma\in\mathlarger\varepsilon_K^{ext}}F_{K,\sigma}(k,k).   
\end{align}
Multiply \eqref{esti6} by $(\eta^+_k)'(u^{n+1}_K)=sign^+(u_K^{n+1}-k)$.
We recall  that for all convex function $J$, we have  for all $z_1, z_2\in\mathbb R$, the convexity inequality $(z_1-z_2)J'(z_1)\geq J(z_1)-J(z_2)$. (Here, we may consider $J'$ as being multivalued, in the sense of sub differential of $J$). First, we use this convexity inequality to obtain
\begin{align}\label{esti8}
sign^+(u_K^{n+1}-k)\biggl((u_K^{n+1}-k)-(u_K^n-k)\biggr)\geq \biggl((u_K^{n+1}-k)^+-(u_K^n-k)^+\biggr).
\end{align}
Second, due to the monotony of the numerical fluxes, we see that
\begin{align}
sign^+(u_K^{n+1}-k)\biggl(F_{K,\sigma}(u_{K}^{n+1}\!,u_{K,\sigma}^{n+})-F_{K,\sigma}(k,k)\biggr)\geq\Phi_{K,\sigma,k}^{+}(u_K^{n+1},u_{K,\sigma}^{n+1}).
\end{align}
Finally, using the convexity inequality and the monotonicity of $\phi$, we have:
\begin{align}
-(\eta^+_k)'(u^{n+1}_K)\biggl[\biggl(\phi(u_{L}^{n+1})-\phi(k)\biggr)-\biggl(\phi(u_{K}^{n+1})-\phi(k)\biggr)\biggr]\geq-\biggl(\eta^+_{\phi(k)}(\phi(u_{L}^{n+1})-\eta^+_{\phi(k)}(\phi(u_{K}^{n+1})\biggr).
\end{align}
Then, we get
\begin{align*}%\label{esti7}
&\frac{1}{\delta t}\biggl((u_K^{n+1}-k)^+-(u_K^n-k)^+\biggr)m(K)+\displaystyle\sum_{\sigma\in\mathlarger\varepsilon_K}\Phi_{K,\sigma,k}^{+}(u_K^{n+1},u_{K,\sigma}^{n+1})\nonumber\\&-\displaystyle\sum_{K|L}\tau_{K|L}\biggl(\eta^+_{\phi(k)}(\phi(u_{L}^{n+1}))-\eta^+_{\phi(k)}(\phi(u_{K}^{n+1}))\biggr)\nonumber\\&\leq\displaystyle\sum_{\sigma\in\mathlarger\varepsilon_K^{ext}}sign^+(u_K^{n+1}-k)m(\sigma)f(k)n_{K,\sigma}.
\end{align*}
This prove \eqref{entropydiscrete}.
In the same way,  we prove the  discrete super-entropy inequalities. Finally, we deduce that $u^{n+1}_K$ satisfies the discrete entropy inequality in this sense:
\begin{align}\label{entropydiscrete01}
&\frac{\eta_k(u^{n+1}_K)-\eta_k(u^{n}_K)}{\delta t}m(K)+\displaystyle\sum_{\sigma\in\bar{\mathlarger\varepsilon}_K}\Phi_{K,\sigma,k}(u_K^{n+1},u_{K,\sigma}^{n+1})\nonumber\\&
-\displaystyle\sum_{K|L}\tau_{K|L}\biggl(\eta_{\phi(k)}(\phi(u_{L}^{n+1})-\eta_{\phi(k)}(\phi(u_{K}^{n+1})\biggr)
\nonumber\\&\leq\displaystyle\sum_{\sigma\in\mathlarger\varepsilon_K^{ext}}sign(u_K^{n+1}-k)m(\sigma)f(k)n_{K,\sigma}.
\end{align}
\end{proof}
%%%%%%%%%%%%%%%%%%%%%%%%%%%%%%%%%%%%%%%%%%%%%%%%%%%%%%%%%%%%%%%%%%%%%%%%%%%%%%%%%%%%%%%%%%%%%%%%%%%%%%%%%%%%%%%%%%%%%%%%%%%%%%%%%%%%%%%%%%%%%%%%%%%%%%%%%%%%%%%%%%%%%
\section{Estimates of discrete solution and existence}
We wish to prove that the approximate solution $\uOdt$ satisfies the continuous entropy inequalities (see section 5). To this purpose, we give fundamental estimates useful for proving convergence of the scheme. First, we prove the $L^\infty$ stability of the scheme, this comes from discrete entropy inequalities and the boundedness of the flux $f$ with the relation \eqref{f}.
                     %%%%%%%%%%%%%%%%%%%%%%%%%%%%%%%%%%%%%%%%%%%%%%%%%%%%%%%%%%%%%%%%%%%%%%%%%%%%%%
\subsection{$L^\infty$ bound on discrete solutions}
\begin{prop}\label{exi}
Suppose that $K\in\mathcal{O}$, the assumptions  \eqref{reg2}, \eqref{monotony}- \eqref{consistency}  hold. Assume that $u_0\in[0,u_{\max}]$. Then the approximate solution $\uOdt(t,x)$ of problem $(P)$ defined by  \eqref{esti0}, \eqref{esti1} satisfies:
\begin{align}\label{linfini}
0\leq u_{K}^n(t,x)\leq u_{\max}\;\;\;\;\;\forall K\in\mathcal{O}.
\end{align}
\end{prop}
\begin{proof}
Summing \eqref{entropydiscrete} over $K\in\mathcal{O}$, we get
\begin{align}\label{entropydiscrete1}
&\displaystyle\sum_{K\in\mathcal{O}}m(K)\frac{\eta^+_k(u^{n+1}_K)-\eta^+_k(u^{n}_K)}{\delta t}+\displaystyle\sum_{K\in\mathcal{O}}\displaystyle\sum_{\sigma\in\bar{\mathlarger\varepsilon}_K}\Phi_{K,\sigma,k}^{+}(u_K^{n+1},u_{K,\sigma}^{n+1})\nonumber\\&-\displaystyle\sum_{K\in\mathcal{O}}\displaystyle\sum_{L\in\mathcal{ N}(K)}\tau_{K|L}\biggl(\eta^+_{\phi(k)}(\phi(u_{L}^{n+1}))-\eta^+_{\phi(k)}(\phi(u_{K}^{n+1}))\biggr)\nonumber\\&\leq\displaystyle\sum_{K\in\mathcal{O}}\displaystyle\sum_{\sigma\in\mathlarger\varepsilon_K^{ext}}m(\sigma)|f(k)n_{K,\sigma}|.
\end{align}
In inequality \eqref{entropydiscrete1}, take $k=u_{\max}$ and use \eqref{f} to obtain:
\begin{align*}
&\displaystyle\sum_{K\in\mathcal{O}}\frac{m(K)}{\delta t}\biggl(u_K^{n+1}-u_{\max}\biggr)^+-\displaystyle\sum_{K\in\mathcal{O}}\frac{m(K)}{\delta t}\biggl(u_K^n-u_{\max}\biggr)^++\displaystyle\sum_{K\in\mathcal{O}}\displaystyle\sum_{\sigma\in\mathlarger\varepsilon_K}\Phi_{K,\sigma,u_{\max}}^{+}(u_K^{n+1},u_{K,\sigma}^{n+1})
\nonumber\\&-\displaystyle\sum_{K\in\mathcal{O}}\sum_{L\in\mathcal{ N}(K)}\tau_{K|L}\biggl(\eta^+_{\phi(u_{\max})}(\phi(u_{L}^{n+1}))-\eta^+_{\phi(u_{\max})}(\phi(u_{K}^{n+1}))\biggr)\leq0.
\end{align*}
From now, remark that due to the conservativity  of the scheme we have
\begin{align*}
&\displaystyle\sum_{K\in\mathcal{O}}\sum_{\sigma\in\mathlarger\varepsilon_K}\!\!\Phi_{K,\sigma,u_{\max}}^{+}(u_K^{n+1},u_{K,\sigma}^{n+1})\!=\displaystyle\sum_{K\in\mathcal{O}}\sum_{L\in\mathcal{N}(K)}\Phi_{K,L,u_{\max}}^{+}(u_K^{n+1},u_{L}^{n+1})=0\\
&\displaystyle\sum_{K\in\mathcal{O}}\!\sum_{K|L}\!\biggl(\eta^+_{\phi(k)}(\phi(u_{L}^{n+1}))-\eta^+_{\phi(k)}(\phi(u_{K}^{n+1}))\biggr)\!=\!\!\displaystyle\sum_{K\in\mathcal{O}}\!\sum_{L\in\mathcal{N}(K)}\!\biggl(\eta^+_{\phi(k)}(\phi(u_{L}^{n+1}))-\eta^+_{\phi(k)}(\phi(u_{K}^{n+1}))\biggr)=0.
\end{align*}
Therefore 
\begin{align}\label{esti11} 
\displaystyle\sum_{K\in\mathcal{O}}\frac{m(K)}{\delta t}\biggl((u_K^{n+1}-u_{\max})^+-(u_K^n-u_{\max})^+\biggr)\leq 0.
\end{align}
Since  $0\leq u_K^0\leq u_{\max}$, by induction we prove $(u_K^{n+1}-u_{\max})^+\leq0$.
In the same way, in the super-entropy inequality, taking $k=0$, use \eqref{f}, we also prove that $(u_K^{n+1})^-\leq0$. 
\end{proof}
               %%%%%%%%%%%%%%%%%%%%%%%%%%%%%%%%%%%%%%%%%%%%%%%%%%%%%%%%%%%%%%%%%%%%%%%%%%%%%%%%%%%%
\subsection{Weak BV and $L^2(0,T,H^1(\Omega))$ estimates}
Now, we give the weak BV and $L^2(0,T,H^1(\Omega))$ estimates. The $L^2(0,T,H^1(\Omega))$ as the $L^\infty$ estimate are necessary for justifying compactness properties of discrete solutions. The weak BV-stability does not give directly any compactness result, however, it plays a crucial role in the proof of continuous entropy inequality (see section 5). To start with, we recall a Lemma which is one ingredient of the proof of Lemma \ref{weakbv} below.
\begin{lem}\label{maj}
Let $G:\mathbb [a,b]\longrightarrow \mathbb R$ be a monotone Lipschitz continuous function with Lipschitz constant $L>0$ and $a,b\in\mathbb R$. Then for all $c,d\in[a,b]$, one has
\begin{align*}
\left|\displaystyle\int_c^d\biggl(G(x)-G(c)\biggr)dx\right|\geq\frac{1}{2L}\biggl(G(d)-G(c)\biggr)^2.
\end{align*}
\end{lem}
\begin{proof}
In order to prove this result, we assume, for instance, that $G$ is nondecreasing and $c<d$ (the other cases are similar). Then, ones has $G(s)\geq H(s)$, for all $s\in[c,d]$, where $H(s)=G(c)$ for $s\in[c,d-l]$ and $H(s)=G(c)+(s-d+l)L$ for $s\in[d-l,d]$, with $lL=G(d)-G(c)$, and therefore:
\begin{align*}
\displaystyle\int_c^d(G(s)-G(c))ds\geq\displaystyle\int_c^d(H(s)-G(c))ds=\frac{l}{2}(G(d)-G(c))=\frac{1}{2L}(G(d)-G(c))^2.
\end{align*}
\end{proof}
Now, we establish the weak BV-stability of the scheme.
\begin{lem}\label{weakbv}(Weak BV-Estimate) 
Suppose that \eqref{reg2},\eqref{monotony}-\eqref{consistency}  hold. Let $\uOdt$ be an approximate solution of problem $(P)$ defined by  \eqref{esti0}, \eqref{esti1}. Let $T>0$, and set $N=\max\{n\in\mathbb N, n<\frac{T}{\delta t}\}$ and $L\in\mathcal N(K)$ (with convention $u^{n+1}_K\geq u^{n+1}_L)$. Then there exists $C=C(||f||_{L^\infty},u_{\max},T,|\Omega|)\geq0$ such that %if $\delta t<T$,%$C=C(Lip(f),T,\alpha,\beta,\beta)\leq 0$ such that $\delta t<T$,
	\begin{align}\label{bvestimate}
	&\displaystyle\sum_{n=0}^N\delta t\displaystyle\sum_{K|L}\biggl[\max_{u_L^{n+1}\leq c\leq d\leq u_K^{n+1}}\biggl(F_{K,\sigma}(d,c)-F_{K,\sigma}(d,d)\biggr)\biggr]\nonumber\\&
	+\displaystyle\sum_{n=0}^N\delta t\displaystyle\sum_{K|L}\biggl[\max_{u_L^{n+1}\leq c\leq d\leq u_K^{n+1}}\biggl(F_{K,\sigma}(d,c)-F_{K,\sigma}(c,c)\biggr)\biggr]\leq \frac{C}{\sqrt{h}}.
	\end{align}
	\end{lem}
	\begin{proof} 
	Multiplying \eqref{esti1} by $\delta t u_K^{n+1}$ and summing over $K\in\mathcal{O}$ and $n=0,...,N$ yields $A_{Evol}+A_{Conv}+A_{Diff}=0$   with
	\begin{align*}
	A_{Evol}&=\displaystyle\sum_{n=0}^N\displaystyle\sum_{K\in\mathcal{O}}m(K)(u^{n+1}_K-u^{n}_K)u^{n+1}_K;\\
	A_{Conv}&=\displaystyle\sum_{n=0}^N\delta t\displaystyle\sum_{K\in\mathcal{O}}\displaystyle\sum_{\sigma\in\mathlarger\varepsilon_K} F_{K,\sigma}(u_{K}^{n+1}\!,u_{K,\sigma}^{n+1}) u_K^{n+1};\\
	A_{Diff}&=-\displaystyle\sum_{n=0}^N\delta t\displaystyle\sum_{K\in\mathcal{O}}\displaystyle\sum_{K|L}\tau_{K|L}\biggl(\phi(u_{L}^{n+1})-\phi(u_{K}^{n+1})\biggr) u_K^{n+1}.
	\end{align*}
Let us first estimate  $A_{Evol}$. We use the fact that:
\begin{align*}
\forall a,b\in\mathbb R,\;\; (a-b)a=\frac{1}{2}(a-b)^2+\frac{1}{2}a^2-\frac{1}{2}b^2,
\end{align*}
 we get:
\begin{align}
A_{Evol}&=\displaystyle\sum_{n=0}^{N}\displaystyle\sum_{K\in\mathcal{O}}m(K)(u^{n+1}_K-u^{n}_K)u^{n+1}_K\nonumber\\&=\frac{1}{2}\displaystyle\sum_{n=1}^{N-1}\displaystyle\sum_{K\in\mathcal{O}}m(K)(u^{n+1}_K-u^{n}_K)^2+\frac{1}{2}\displaystyle\sum_{K\in\mathcal{O}}m(K)\biggl[(u^{N+1}_K)^2-(u^{0}_K)^2\biggr].
\end{align}
The two first terms are non negative  and due to \eqref{linfini} there exists $C\geq 0$ (that only depends on $|\Omega|$ and $u_{\max}$) such that $-C$ is a lower bound for the last term, then
\begin{align}\label{estiA}
A_{Evol}\geq-C.
\end{align}
Secondly, using summation by parts
\begin{align}
A_{Diff}=\displaystyle\sum_{n=0}^N\delta t\displaystyle\sum_{K\in\mathcal{O}}\displaystyle\sum_{K|L}\tau_{K|L}\biggl(\phi(u_{L}^{n+1})-\phi(u_{K}^{n+1})\biggr) \biggl(u_L^{n+1}-u_K^{n+1}\biggr)\geq0.
\end{align}
Now, we study the term $A_{Conv}$. Due to \eqref{Kext}, it can be rewritten as the sum between $A_{Conv}^{int}$ and $A_{Conv}^{ext}$:
\begin{align}
A_{Conv}^{int}&=\displaystyle\sum_{n=0}^N\displaystyle\sum_{K|L}\delta t\biggl(F_{K,\sigma}(u_{K}^{n+1}\!,u_{L}^{n+1})-F_{K,\sigma}(u_{K}^{n+1}\!,u_{K}^{n+1})\biggr) u_K^{n+1}\nonumber\\
&-\displaystyle\sum_{n=0}^N\displaystyle\sum_{K|L}\delta t\biggl(F_{K,\sigma}(u_{K}^{n+1}\!,u_{L}^{n+1})-F_{K,\sigma}(u_{L}^{n+1}\!,u_{L}^{n+1})\biggr) u_L^{n+1};
\end{align}
\begin{align}
A_{Conv}^{ext}&=-\displaystyle\sum_{n=0}^N\delta t\displaystyle\sum_{K\in\mathcal{O}}u_K^{n+1}\displaystyle\sum_{\sigma\in\mathlarger\varepsilon_K^{ext}}F_{K,\sigma}(u_{K}^{n+1}\!,u_{K}^{n+1})\nonumber\\
&=-\displaystyle\sum_{n=0}^N\delta t\displaystyle\sum_{K\in\mathcal{O}}u_K^{n+1}\displaystyle\sum_{\sigma\in\mathlarger\varepsilon_K^{ext}}m(\sigma) f(u^{n+1}_K).\eta_{K,\sigma}.\nonumber
\end{align}
We can estimate the boundary term $A_{Conv}^{ext}$ by
\begin{align}
|A_{Conv}^{ext}|\leq C(||f||_{L^\infty},u_{\max},T,|\partial\Omega|).
\end{align}
Let us assign:
\begin{align*}
\Psi_{K,L}(a)=\displaystyle\int_0^a s\biggl(\frac{\partial F_{K,L}}{\partial u}(s,s)+\frac{\partial F_{K,L}}{\partial v}(s,s)\biggr)ds=\displaystyle\int_0^a s\frac{d}{ds}F_{K,L}(s,s)ds.
\end{align*}
Then
\begin{align}\label{dF}
\Psi_{K,L}(b)-\Psi_{K,L}(a)&=\displaystyle\int_0^b s\frac{d}{ds}F_{K,L}(s,s)ds-\displaystyle\int_0^a s\frac{d}{ds}F_{K,L}(s,s)ds,\nonumber\\
&=b\biggl(F_{K,L}(b,b)-F_{K,L}(a,b)\biggr)-a\biggl(F_{K,L}(a,a)-F_{K,L}(a,b)\biggr)\nonumber\\&-\int^b_a\biggl(F_{K,L}(s,s)-F_{K,L}(a,b)\biggr)ds.
\end{align}
Take $a=u^{n+1}_{K}$ and $b=u_L^{n+1}$ in \eqref{dF} and multiply by $\delta t$. Summing over $n=0,....,N$  and $L\in\mathcal{N}(K)$, we obtain $A_{Conv}^{int}=A_{Conv}^{int,1}+A_{Conv}^{int,2}$, where:
\begin{align*}
A_{Conv}^{int,1}&=\displaystyle\sum_{n=0}^{n+1}\displaystyle\sum_{K|L}\delta t \int^{u^{n+1}_{K}}_{u^{n+1}_{L}}\biggl(F_{K,L}(u^{n+1}_{K},u^{n+1}_{L})-F_{K,L}(s,s)\biggr)ds\nonumber\\
A_{Conv}^{int,2}&=-\displaystyle\sum_{n=0}^{n+1}\displaystyle\sum_{K|L}\delta t\biggl(\Psi_{K,L}(u_K^{n+1})-\Psi_{K,L}(u_L^{n+1})\biggr).
\end{align*}
We have $\displaystyle\sum_{\sigma\in\bar{\mathlarger\varepsilon}_K}\Psi_{K,L}(s)=0$, for all $s\in [0,u_{\max}]$; then it appears that $A_{Conv}^{int,2}$ reduces to the sum of $\sigma\in\mathlarger\varepsilon_K^{ext}$, and it satisfies $|A_{Conv}^{int,2}|\leq C$ similar to the estimate of $A_{Conv}^{ext}$.\\ 
Now, consider $a,b,c,d\in\mathbb R$ such that $a\leq c\leq d\leq b$. Using the monotonicity of $F_{K,L}$ and Lemma \ref{maj}, we deduce
\begin{align}
\int^b_a\biggl(F_{K,L}(b,a)-F_{K,L}(s,s)\biggr)ds\geq&\int^c_d\biggl(F_{K,L}(d,c)-F_{K,L}(d,s)\biggr)\nonumber\\&\geq\frac{1}{2m(K|L)M}\biggl(F_{K,L}(d,c)-F_{K,L}(d,d)\biggr)^2,
\end{align}
\begin{align}
\int^b_a\biggl(F_{K,L}(b,a)-F_{K,L}(s,s)\biggr)ds\geq&\int^c_d\biggl(F_{K,L}(d,c)-F_{K,L}(d,s)\biggr)\nonumber\\&\geq\frac{1}{2m(K|L)M}\biggl(F_{K,L}(d,c)-F_{K,L}(c,c)\biggr)^2.
\end{align}
Therefore, we get:
\begin{align*}
&\int^{u^{n+1}_K}_{u^{n+1}_L}\!\!\!\biggl(F_{K,L}({u^{n+1}_K},{u^{n+1}_L})\!-\!F_{K,L}(s,s)\biggr)ds\geq\frac{1}{2m(K|L)M}\max_{{u^{n+1}_L}\leq c\leq d\leq {u^{n+1}_K}}\biggl(F_{K,L}(d,c)\!-\!F_{K,L}(d,d)\biggr)^2
\end{align*}
\begin{align*}
&\int^{u^{n+1}_K}_{u^{n+1}_L}\!\!\!\biggl(F_{K,L}({u^{n+1}_K},{u^{n+1}_L})\!-\!F_{K,L}(s,s)\biggr)ds\geq\frac{1}{2m(K|L)M}\max_{{u^{n+1}_L}\leq c\leq d\leq {u^{n+1}_K}}\biggl(F_{K,L}(d,c)\!-\!F_{K,L}(c,c)\biggr)^2.
\end{align*} 
Then, we have 
 \begin{align}
A_{Conv}^{int,1}&\geq\displaystyle\sum_{n=0}^{N}\delta t\displaystyle\sum_{K|L}\frac{1}{4m(K|L)M}\biggl[\max_{u_L^{n+1}\leq c\leq d\leq u_K^{n+1}}\biggl(F_{K,\sigma}(d,c)-F_{K,\sigma}(d,d)\biggr)^2\biggr]\nonumber\\&
	+\displaystyle\sum_{n=0}^{N}\delta t\displaystyle\sum_{K|L}\frac{1}{4m(K|L)M}\biggl[\max_{u_L^{n+1}\leq c\leq d\leq u_K^{n+1}}\biggl(F_{K,\sigma}(d,c)-F_{K,\sigma}(c,c)\biggr)^2\biggr].
\end{align}
Recalling the equality $A_{Conv}=A_{Conv}^{ext}+A_{Conv}^{int,1}+A_{Conv}^{int,2}$, we find
\begin{align}\label{estiB}
&A_{Conv}\geq\displaystyle\sum_{n=0}^{N}\delta t\displaystyle\sum_{K|L}\frac{1}{4m(K|L)M}\biggl[\max_{u_L^{n+1}\leq c\leq d\leq u_K^{n+1}}\biggl(F_{K,\sigma}(d,c)-F_{K,\sigma}(d,d)\biggr)^2\biggr]\nonumber\\&
	+\displaystyle\sum_{n=0}^{N}\delta t\displaystyle\sum_{K|L}\frac{1}{4m(K|L)M}\biggl[\max_{u_L^{n+1}\leq c\leq d\leq u_K^{n+1}}\biggl(F_{K,\sigma}(d,c)-F_{K,\sigma}(c,c)\biggr)^2\biggr]-2C.
\end{align}
Set
\begin{align*}
\bar{A}_{Conv}&=\displaystyle\sum_{n=0}^{N}\delta t\displaystyle\sum_{K|L}\frac{1}{4m(K|L)M}\biggl[\max_{u_L^{n+1}\leq c\leq d\leq u_K^{n+1}}\biggl(F_{K,\sigma}(d,c)-F_{K,\sigma}(d,d)\biggr)^2\biggr]\nonumber\\&+\displaystyle\sum_{n=0}^{N}\delta t\displaystyle\sum_{(K,L)\in\mathlarger\varepsilon^{n+1}_{int}}\frac{1}{4m(K|L)M}\biggl[\max_{u_L^{n+1}\leq c\leq d\leq u_K^{n+1}}\biggl(F_{K,\sigma}(d,c)-F_{K,\sigma}(c,c)\biggr)^2\biggr].
\end{align*}
Now, as the equality $A_{Evol}+A_{Conv}+A_{Diff}=0$ holds and as \eqref{estiA} and \eqref{estiB} are satisfied, we have  $\bar{A}_{Conv}\leq C$.\\
Moreover, using the Cauchy-Schwarz inequality, we deduce 
\begin{align}\label{estiB}
&\displaystyle\sum_{n=0}^{N}\delta t\displaystyle\sum_{K|L}\biggl[\max_{u_L^{n+1}\leq c\leq d\leq u_K^{n+1}}\biggl(F_{K,\sigma}(d,c)-F_{K,\sigma}(d,d)\biggr)\biggr]\nonumber\\&
	+\displaystyle\sum_{n=0}^{N}\delta t\displaystyle\sum_{K|L}\biggl[\max_{u_L^{n+1}\leq c\leq d\leq u_K^{n+1}}\biggl(F_{K,\sigma}(d,c)-F_{K,\sigma}(c,c)\biggr)\biggr]\nonumber\\&\leq\frac{1}{4M}\sqrt{\displaystyle\sum_{n=0}^{N}\delta t\displaystyle\sum_{K|L}m(K|L)}\sqrt{\bar{A}_{Conv}}.
\end{align}
At the end, take into account the regularity on the mesh   \eqref{reg2} to deduce that:
\begin{align*}
\displaystyle\sum_{n=0}^N\delta t\displaystyle\sum_{K|L}m(K|L)\leq T\frac{1}{\alpha}h^{\ell-1}\displaystyle\sum_{K|L}1\leq T\frac{1}{\alpha}h^{\ell-1}m(\Omega)\frac{1}{\alpha}h^{-\ell}\leq\frac{C}{h}.
\end{align*}
\end{proof}
\begin{lem}\label{H1estimate}($L^2( 0,T,H^1(\Omega))$ Estimate) 
Suppose that \eqref{reg2},\eqref{monotony}-\eqref{consistency}  hold. Let $\uOdt$ be the approximate solution of problem $(P)$ defined by  \eqref{esti0}, \eqref{esti1}, \eqref{neumann}. Let $T>0$, and set $N=\max\{n\in\mathbb N, n<\frac{T}{\delta t}\}$. Then there exists $C=C(||f||_{L^\infty},u_{\max},T)\geq0$ such that% if $\delta t<T$,
\begin{align}\label{H1estima}
\frac{1}{2}\displaystyle\sum_{n=0}^N\delta t\displaystyle\sum_{K\in\mathcal{O}}\displaystyle\sum_{L\in\mathcal{N}(K)}\tau_{K|L}\biggl|\phi(u_{K}^{n+1})-\phi(u_{L}^{n+1})\biggr|^2	\leq C.
\end{align}
\end{lem}
\begin{proof}
Multiplying \eqref{esti1} by $\delta t \phi(u_K^{n+1})$ and summing over $K\in\mathcal{O}$ and $n=0,...,N$ yields $B_{Evol}+B_{Conv}+B_{Diff}=0$   with 
\begin{align*}
	B_{Evol}&=\displaystyle\sum_{n=0}^N\displaystyle\sum_{K\in\mathcal{O}}m(K)(u^{n+1}_K-u^{n}_K)\phi(u_K^{n+1}),\\
	B_{Diff}&=-\displaystyle\sum_{n=0}^N\delta t\displaystyle\sum_{K\in\mathcal{O}}\displaystyle\sum_{K|L}\tau_{K|L}\biggl(\phi(u_{L}^{n+1})-\phi(u_{K}^{n+1})\biggr) \phi(u_K^{n+1}),\\
	B_{Conv}&=\displaystyle\sum_{n=0}^N\delta t\displaystyle\sum_{K\in\mathcal{O}}\displaystyle\sum_{\sigma\in\mathlarger\varepsilon_K}F_{K,\sigma}(u_{K}^{n+1}\!,u_{K,\sigma}^{n+1})\phi(u_K^{n+1}).
\end{align*}
	Let $\vartheta(r)=\int_0^r\phi(s)ds$. From the convexity inequality, we have:
\begin{align}
B_{Evol}&\geq \displaystyle\sum_{n=0}^N\displaystyle\sum_{K\in\mathcal{O}}m(K)\biggl(\vartheta(u_K^{n+1})-\vartheta(u_K^{n})\biggr)\nonumber\\&=\displaystyle\sum_{K\in\mathcal{O}}m(K)\biggl(\vartheta(u_K^{N+1})-\vartheta(u_K^{0})\biggr).
\end{align}
Further, in the term $B_{Diff}$, for every edge $K|L$ the terms involving $K$ and $L$ appear twice. Thanks to the conservativity of the scheme, we find
\begin{align}\label{H1esti}
B_{Diff}&=-\displaystyle\sum_{n=0}^N\delta t\displaystyle\sum_{K\in\mathcal{O}}\displaystyle\sum_{L\in\mathcal{N}(K)}\tau_{K|L}\biggl|\phi(u_{K}^{n+1})-\phi(u_{L}^{n+1})\biggr|^2.	
\end{align}
The term $B_{Conv}$ can be rewritten as
\begin{align*}
B_{Conv}&=\displaystyle\sum_{n=0}^N\delta t\displaystyle\sum_{K\in\mathcal{O}}\displaystyle\sum_{L\in \mathcal{N}(K)}\delta tF_{K,K|L}(u_{K}^{n+1},u_{L}^{n+1})\biggl(\phi(u_L^{n+1})-\phi(u_K^{n+1})\biggr).\\
\end{align*}
Using the weighted Young inequality and \eqref{borne}, we deduce
\begin{align*}
|B_{Conv}|&\leq\displaystyle\sum_{n=0}^N\delta t\displaystyle\sum_{K\in\mathcal{O}}\displaystyle\sum_{L\in\mathcal{N}(K)}\frac{d_{K,L}}{2m(K|L)}\biggl(F_{K,\sigma}(u_{K}^{n+1},u_{L}^{n+1})\biggr)^2\\
&+\frac{1}{2}\displaystyle\sum_{n=0}^N\delta t\displaystyle\sum_{K\in\mathcal{O}}\displaystyle\sum_{L\in\mathcal{N}(K)}\frac{m(K|L)}{d_{K,L}}\biggl|\phi(u_{K}^{n+1})-\phi(u_{L}^{n+1})\biggr|^2\\
&\leq C\displaystyle\sum_{L|K}  m(\widehat{K|L})+\frac{1}{2}\displaystyle\sum_{n=0}^N\delta t\displaystyle\sum_{K\in\mathcal{O}}\displaystyle\sum_{L\in\mathcal{N}(K)}\tau_{K|L}\biggl|\phi(u_{K}^{n+1})-\phi(u_{L}^{n+1})\biggr|^2.
\end{align*}
Collecting the previous inequalities we readily deduce \eqref{H1estima}. This concludes the proof of the Lemma \ref{H1estimate}.
\end{proof} 
                          %%%%%%%%%%%%%%%%%%%%%%%%%%%%%%%%%%%%%%%%%%%%%%%%%%%%%%%%%%%%%%%%%%%%%%%%%%%%%%%%
\subsection{Estimates of space and time translates}
Recall the following result.
\begin{thm}\label{kolmogorov}(\normalfont{Riesz-Frechet-Kolmogorov})
Let an open $Q\subset\mathbb R^{\ell+1}$ and let $\omega\subset\subset Q$. Consider $\mathcal{K}$  a bounded set of $L^p$, with $1\leq p<\infty$. we suppose that: $\forall\epsilon>0$, there exists $\delta>0$, $\delta<dist(\omega,\mathbb R^{\ell+1}\backslash Q)$ such that $||f(x+h)-f(x)||_{L^p(\omega)}\leq\epsilon$ $\forall h\in\mathbb R^{\ell+1}$ with $|h|<\delta$ and $\forall f\in\mathcal{K}$. Then $\mathcal{K}$ is relatively compact in $L^p(\omega)$.  
\end{thm}
Now, we derive estimates of space and time translates of the function $\phi(\uOdt)$ which imply that the sequence is relatively compact in $L^2(Q)$.\\
Notice that because $(\phi(\uOdt)_{\mathcal{O},\delta t}$ obey a uniform $L^\infty$ bound, the local compactness in $Q$ is enough to deduce the $L^2$ compactness. 
\begin{lem}\label{TRANSLATION}
 Let, $\uOdt$ be the approximate solution of problem $(P)$ defined by  \eqref{esti0}, \eqref{esti1}. There exists  a constant $C_1$ depending on $\Omega$, $T$, $|\phi|_{H_{\mathcal{O}}}$ that
\begin{align}\label{translate01}
\displaystyle\int_0^T\displaystyle\int_{\Omega_\eta}\biggl|\phi(\uOdt(t,x+\eta))-\phi(\uOdt(t,x))\biggr|^2 dxdt	\leq C_1|\eta|
\end{align}
for all $\eta\in\mathbb R^l,$ where $\Omega_\eta=\Bigl\{x\in\Omega, [x+\eta,x]\subset\Omega\Bigr\}$
and there exists  $C_2$ depending on $\Omega$, $T$, $\phi$, $f$ such that
\begin{align}\label{translate02}
\forall \tau>0,\;\;\displaystyle\int_0^{T-\tau}\displaystyle\int_{\Omega}\biggl|\phi(\uOdt(t+\tau,x))-\phi(\uOdt(t,x))\biggr|^2dxdt\leq C_2\tau%\sqrt{\tau}
\end{align}
for all $\tau\in(0,T)$.
\end{lem}
\begin{proof}
$\bullet$  First, we prove \eqref{translate01}\\
Let $\eta\in\mathbb R^l$ with $\eta\neq0$ and set $\Omega_\eta=\Bigl\{x\in\Omega, [x+\eta,x]\subset\Omega\Bigr\}$. For all $K\in\mathcal{O}$ and $\sigma\in\mathlarger\varepsilon_K$, define $\chi_\sigma:\Omega_\eta\times\Omega_\eta\longrightarrow\{0,1\}$ by $\chi_\sigma(x,y)=1$ if $[x,y]\cap\sigma\neq\emptyset$  else $\chi_\sigma(x,y)=0$.  One has
\begin{align}\label{space1}
\biggl|\phi(\uOdt(t,x+\eta))-\phi(\uOdt(t,x))\biggr|\leq\displaystyle\sum_{L\in\mathcal{N}(K)}\chi_{K|L}(x,x+\eta)|D_{K|L}\phi|;\;\mbox{ for a.e. } x\in\Omega_\eta
\end{align}
where $D_{K|L}\phi$ is defined as
\begin{align*}
D_{K|L}\phi=|\phi(u^{n+1}_K)-\phi(u^{n+1}_L)|.
\end{align*}
We integrate \eqref{space1} over $\Omega_\eta$, and get:
\begin{align}\label{space2}
\displaystyle\int_{\Omega_\eta}\biggl|\phi(\uOdt(t,x+\eta))-\phi(\uOdt(t,x))\biggr|dx\leq\displaystyle\sum_{L\in\mathcal{N}(K)}\displaystyle\int_{\Omega_\eta}\chi_{K|L}(x,x+\eta)|D_{K|L}\phi|dx.
\end{align}
Remark that, for all $\sigma=K|L\in\mathlarger\varepsilon_K$, $\int_{\Omega_\eta}\chi_{K|L}(x,x+\eta)dx$ is the measure of the set of points of $\Omega$ which are located inside the cylinder whose basis is $K|L$ and generator vector is $-\eta$. Thus
\begin{align}
\int_{\Omega_\eta}\chi_{K|L}(x,x+\eta)dx\leq m(K|L)|\eta|.
\end{align} 
 The relation \eqref{space2} gives
\begin{align}\label{space3}
\displaystyle\int_{\Omega_\eta}\biggl|\phi(\uOdt(t,x+\eta))-\phi(\uOdt(t,x))\biggr|dx&\leq|\eta|\displaystyle\sum_{L\in\mathcal{N}(K)}m(K|L)|D_{K|L}\phi|\nonumber\\
&\leq |\eta|\displaystyle\sum_{L\in\mathcal{N}(K)}m(K|L)|D_{K|L}\phi|.
\end{align}
From now, integrate \eqref{space3} over $[0,T]$
\begin{align}\label{space4}
\displaystyle\int_0^T\displaystyle\int_{\Omega_\eta}\biggl|\phi(\uOdt(t,x+\eta))-\phi(\uOdt(t,x))\biggr|dxdt&\leq|\eta|\sum_{n=0}^{N}\delta t\displaystyle\sum_{L\in\mathcal{N}(K)}m(K|L)d_{K,L}\biggl|\frac{D_{K|L}\phi}{d_{K,L}}\biggr|.
\end{align}
Remark that:
\begin{align}\label{space5}
\biggl|\biggl|\phi(\uOdt(t,x+\eta))-\phi(\uOdt(t,x))\biggr|\biggr|_{L^2(Q_\eta)}^2&\leq2||\phi||_{L^\infty}\biggl|\biggl|\phi(\uOdt(t,x+\eta))-\phi(\uOdt(t,x))\biggr|\biggr|_{L^1(Q_\eta)}.
\end{align}
Then \eqref{space4} and \eqref{space5} give
\begin{align}\label{space5}
\displaystyle\int_0^T\displaystyle\int_{\Omega_\eta}\biggl|\phi(\uOdt(t,x+\eta))-\phi(\uOdt(t,x))\biggr|^2dxdt\leq C_1|\eta|.
\end{align}
%with $C_1=\sqrt{T|\Omega_\eta|}|\phi|_{H_{\mathcal{O}}}$\\.% where $|.|_{H_{\mathcal{O}}}$ is  the $H^1_0$ discrete semi norms defined by \eqref{H1seminorm}.\\
$\bullet$ Finally,  we prove \eqref{translate02}.\\
Let $\tau\in(0,T)$ and $t\in(0,t-\tau)$. Set $n_0=[t/\delta t]$ and $n_1=[(t+\tau)/\delta t]$, let 
\begin{align*}
I^{t,\tau}&=\{n\in\mathbb N, \mbox{ such that } t<(n+1)\delta t\leq t+\tau \}\\
J^{t,\tau}&=\{n\in\mathbb N, \mbox{ such that } (n+1)\delta t-\tau\leq t< (n+1)\delta t\}.
\end{align*}
Since $\phi$ is Locally continuous with constant $\phi_{Lip}=\displaystyle\sup_{0<a< b\leq u_{\max}}\frac{\phi(a)-\phi(b)}{a-b}$, one has:
\begin{align}
\displaystyle\int_0^{T-\tau}\displaystyle\int_{\Omega}\biggl|\phi(\uOdt(t+\tau,x))-\phi(\uOdt(t,x))\biggr|^2	dxdt\leq \phi_{Lip}\displaystyle\int_0^{T-\tau}S(t)dt
\end{align} 
where, for almost every $t\in(0,T-\tau)$
\begin{align*}
S(t)&=\displaystyle\int_{\Omega}\biggl(\phi(\uOdt(t+\tau,x))-\phi(\uOdt(t,x))\biggr)\biggl(\uOdt(t+\tau,x)-\uOdt(t,x))\biggr)dx	dxdt\\
&=\displaystyle\sum_{K\in\mathcal{O}}m(K)\biggl(\phi(u_K^{n_1})-\phi(u_K^{n_0})\biggr)\biggl(u_K^{n_1}-u_K^{n_0}\biggr)\\
&=\displaystyle\sum_{K\in\mathcal{O}}\biggl(\phi(u_K^{n_1})-\phi(u_K^{n_0})\biggr)\displaystyle\sum_{I^{t,\tau}}m(K)\biggl(u_K^{n+1}-u_K^{n}\biggr).
\end{align*} 
Use \eqref{esti1} and gather by edges. We get
\begin{align*}
S(t)&=\displaystyle\sum_{K\in\mathcal{O}}\biggl(\phi(u_K^{n_1})-\phi(u_K^{n_0})\biggr)\displaystyle\sum_{I^{t,\tau}}\delta t\biggl[\displaystyle\sum_{L\in\mathcal{N}(K)}\frac{m(K|L)}{d_{K,L}}\biggl(\phi(u_{L}^{n+1})-\phi(u_{K}^{n+1})\biggr)-F_{K,L}(u_{K}^{n+1}\!,u_{L}^{n+1})\biggr]\\
&=\displaystyle\sum_{I^{t,\tau}}\delta t\displaystyle\sum_{K|L}\frac{m(K|L)}{d_{K,L}}\biggl(\phi(u_K^{n_1})-\phi(u_L^{n_1})-\phi(u_K^{n_0})+\phi(u_L^{n_0})\biggr)\biggl(\phi(u_{L}^{n+1})-\phi(u_{K}^{n+1})\biggr)\\&+\displaystyle\sum_{I^{t,\tau}}\delta t\displaystyle\sum_{K|L}\biggl(\phi(u_K^{n_1})-\phi(u_L^{n_1})-\phi(u_K^{n_0})+\phi(u_L^{n_0})\biggr)F_{K,L}(u_{L}^{n+1}\!,u_{K}^{n+1}).
\end{align*} 
We can then use the inequality $2ab\leq a^2+b^2$. We get $S(t)\leq \frac{1}{2}S_0(t)+\frac{1}{2}S_1(t)+S_2(t)+S_3(t)+S_4$ with:
\begin{align*}
S_0(t)&=\displaystyle\sum_{I^{t,\tau}}\delta t\displaystyle\sum_{K|L}\frac{m(K|L)}{d_{K,L}}\biggl(\phi(u_K^{n_0})-\phi(u_L^{n_0})\biggr)^2\\
S_1(t)&=\displaystyle\sum_{I^{t,\tau}}\delta t\displaystyle\sum_{K|L}\frac{m(K|L)}{d_{K,L}}\biggl(\phi(u_K^{n_1})-\phi(u_L^{n_1})\biggr)^2\\
S_2(t)&=\displaystyle\sum_{I^{t,\tau}}\delta t\displaystyle\sum_{K|L}\frac{m(K|L)}{d_{K,L}}\biggl(\phi(u_K^{n+1})-\phi(u_L^{n+1})\biggr)^2\\
S_3(t)&=\displaystyle\sum_{I^{t,\tau}}\delta t\displaystyle\sum_{K|L}\biggl(\phi(u_K^{n_1})-\phi(u_L^{n_1})\biggr)F_{K,L}(u_{L}^{n+1}\!,u_{K}^{n+1})\\
S_4(t)&=\displaystyle\sum_{I^{t,\tau}}\delta t\displaystyle\sum_{K|L}\biggl(\phi(u_L^{n_0})-\phi(u_K^{n_0})\biggr)F_{K,L}(u_{L}^{n+1}\!,u_{K}^{n+1}).
\end{align*} 
We introduce the function $\chi^t$ such that $\chi^t(1)=1$ and $\chi^t(0)=0$. We have, for all $t\in\mathbb R^+$ and $n\in\mathbb N$, $\chi^t(I^{t,\tau})=\chi^t(J^{t,\tau})$. Therefore
\begin{align}
\displaystyle\int_0^{T-\tau}S_0(t)dt&\leq\displaystyle\sum_{n=0}^{[T/\delta t]}\delta t\displaystyle\sum_{K|L}\frac{m(K|L)}{d_{K,L}}\biggl(\phi(u_K^{n_0})-\phi(u_L^{n_0})\biggr)^2\displaystyle\int_{n_0\delta t}^{(n_0+n)\delta t+\tau}\displaystyle\sum_{n\in\mathbb N}\chi^t(I^{t,\tau})dt\nonumber\\&
\leq\displaystyle\sum_{n=0}^{[T/\delta t]}\delta t\displaystyle\sum_{K|L}m(K|L)d_{K,L}\biggl|\frac{\phi(u_K^{n_0})-\phi(u_L^{n_0})}{d_{K,L}}\biggr|^2\displaystyle\int_{n_0\delta t}^{(n_0+n)\delta t+\tau}\displaystyle\sum_{n\in\mathbb N}\chi^t(I^{t,\tau})dt\nonumber\\&
\leq\displaystyle\sum_{n=0}^{[T/\delta t]}\delta t\displaystyle\sum_{K|L}\ell m(\widehat{K|L})\biggl|\frac{\phi(u_K^{n_0})-\phi(u_L^{n_0})}{d_{K,L}}\biggr|^2\displaystyle\int_{n_0\delta t}^{(n_0+n)\delta t+\tau}\displaystyle\sum_{n\in\mathbb N}\chi^t(I^{t,\tau})dt.
\end{align}
Notice the following  property:
\begin{align} 
\displaystyle\int_{n_0\delta t}^{(n_0+1)\delta t}\displaystyle\sum_{n\in\mathbb N}\chi^t(J^{t,\tau})dt=\displaystyle\sum_{n\in\mathbb N}\displaystyle\int_{(n_0-n-1)\delta t+\tau}^{(n_0-n)\delta t+\tau}\chi^t(0\leq t<\tau)dt=\tau.
\end{align}
Using \eqref{H1estima}, we find:
\begin{align}\label{Ttrans0}
\displaystyle\int_0^{T-\tau}S_0(t)dt\leq C\tau.                                              
\end{align}
We get in the same way
\begin{align}\label{Ttrans1}
\displaystyle\int_0^{T-\tau}S_1(t)dt\leq C\tau.
\end{align}
We now turn to the study of the third term:
\begin{align*}
\displaystyle\int_0^{T-\tau}S_2(t)dt&\leq\displaystyle\sum_{n=0}^{[T/\delta t]}\delta t\displaystyle\sum_{K|L}\frac{m(K|L)}{d_{K,L}}\biggl(\phi(u_K^{n+1})-\phi(u_L^{n+1})\biggr)^2\displaystyle\int_0^{T-\tau}\chi^t(J^{t,\tau})dt\nonumber\\&
\leq\displaystyle\sum_{n=0}^{[T/\delta t]}\delta t\displaystyle\sum_{K|L}\ell m(\widehat{K|L})\biggl|\frac{\phi(u_K^{n+1})-\phi(u_L^{n+1})}{d_{K,L}}\biggr|^2\displaystyle\int_{n_0\delta t}^{(n_0+n)\delta t+\tau}\displaystyle\sum_{n\in\mathbb N}\chi^t(J^{t,\tau})dt.
\end{align*}
Because
\begin{align*}
\displaystyle\int_0^{T-\tau}\chi^t(J^{t,\tau})dt=\min(T-\tau,(n+1)\delta t)-\max(0,(n+1)\delta t-\tau)\leq \tau,
\end{align*}
we get 
\begin{align}\label{Ttrans2}
\displaystyle\int_0^{T-\tau}S_2(t)dt\leq C\tau.
\end{align}
Recall that due to \eqref{borne}
\begin{align}\label{diam}
\biggl|\frac{F_{K,L}(a,b)}{m(K|L)}\biggr|\leq(||f||_{L^\infty}+M).
\end{align}
We have in the same way
\begin{align}\label{Ttrans3}
\displaystyle\int_0^{T-\tau}S_3(t)dt&\leq\displaystyle\sum_{n=0}^{[T/\delta t]}\delta t\displaystyle\sum_{K|L}\biggl(\phi(u_K^{n_1})-\phi(u_L^{n_1})\biggr)F_{K,L}(u_{L}^{n+1}\!,u_{K}^{n+1})\displaystyle\int_0^{T-\tau}\chi^t(J^{t,\tau})dt\nonumber\\
&\leq\displaystyle\sum_{n=0}^{[T/\delta t]}\;\delta t\displaystyle\;\sum_{K|L}m(K|L)d_{K,L}\;\biggl|\frac{\phi(u_K^{n_1})-\phi(u_L^{n_1})}{d_{K,L}}\biggr|\biggl|\frac{F_{K,L}(u_{L}^{n+1}\!,u_{K}^{n+1})}{m(K|L)}\biggr|\tau\nonumber\\
&\leq\biggl((||f||_{L^\infty}+M)\displaystyle\sqrt{Tm(\Omega)}|\phi(u)|_{H_{\mathcal{O}}}\biggr)\tau.
\end{align}
In the same way we prove:
\begin{align}\label{Ttrans4}
\displaystyle\int_0^{T-\tau}S_4(t)dt\leq\biggl((||f||_{L^\infty}+M)\sqrt{Tm(\Omega)}|\phi(u)|_{H_{\mathcal{O}}}\biggr)\tau.
\end{align}
From \eqref{Ttrans0}, \eqref{Ttrans1}, \eqref{Ttrans2}, \eqref{Ttrans3} and \eqref{Ttrans4}, we get:
\begin{align}
\displaystyle\int_0^{T-\tau}\displaystyle\int_{\Omega}|\phi(\uOdt(t+\tau,x))-\phi(\uOdt(t,x))|^2dxdt\leq C_2\tau.
\end{align}
\end{proof}
                                     %%%%%%%%%%%%%%%%%%%%%%%%%%%%%%%%%%%%%%%%%%%%%%%%%%%%%%%%%%%%%%%%%%%%%%%%%%%%%%%%%%%%%
\subsection{Existence of a discrete solution}
The proof of existence for the scheme \eqref{esti0}, \eqref{esti1} is obtained by applying the Leray-Schauder topological degree theorem. The idea is to modify continuously the scheme to obtain a system which admits a solution and if the modification preserves in the same time the estimates (in our case this can get easily by the $L^\infty$ norm on $u_{\mathcal{O},\delta t}$), then the scheme also has a solution.
\begin{defn}\label{Leray} 
Let $E$ be a real Banach space. We denote by $\mathcal{A}$ the set of $(Id-g,B,y)$ where $g: \bar{B}\longrightarrow E$ is a compact with $B\subset E$ and $y\in E$ such that $y\notin\{g(x),x\in\partial B\}$. 
\end{defn}
%\begin{thm}\label{Schauder}(Leray, Schauder)
%Let $E$ a real Banach space. Consider $\mathcal{A}$ in the sense of Definition \ref{Leray}. There exist an application  $degree:\mathcal{A}\longrightarrow\mathbb Z$  called topological degree, which satisfies:\\
%$\bullet$ $degree(Id,B,y)=1$ if $y\in B$.\\
%$\bullet$ $degree(Id-g,B,y)=degree(Id,B_1,y)+degree(Id,B_2,y)$ if $B_1\cup B_2\subset B$, $B_1\cap B_2=\emptyset$ and $y\notin\{x-g(x),x\in\bar{B}\backslash B_1\cup B_2\}$.\\
%$\bullet$ If $h:[0,1]\times\bar{B}\longrightarrow E$ is compact, $y\in \mathcal{C}([0,1];E)$ and $y(\lambda)\notin\{x-h(t,x),x\in B\}$ ( for all $\lambda\in[0,1]$), then $degree(Id-h(\lambda,.),B,y(\lambda))=degree(Id-h(0,.),B,y(0))$ for all $\lambda\in[0,1]$.
%\end{thm}
\begin{lem}
Suppose \eqref{f} is satisfied. Then for all $K\in\mathcal{O}$, there exist $u_K^{n+1}$ satisfying \eqref{entropydiscrete}. 
\end{lem}
\begin{proof}
For the proof, we consider for every $\alpha\in[0,1]$ the following problem:
\begin{equation*}
\left\{\begin{array}{rll}
&v_K-\alpha\biggl[u_K^0-\frac{\delta t}{m(K)}\displaystyle\sum_{\sigma\in\mathlarger\varepsilon_K}F_{K,\sigma}(v_{K}\!,v_{K,\sigma})+\frac{\delta t}{m(K)}\displaystyle\sum_{\sigma\in\mathlarger\varepsilon_K}\tau_{K,\sigma}\biggl(\phi(v_{K,\sigma})-\phi(v_{K})\biggr)\biggr]=0\\
&\forall K\in\mathcal{O},
\end{array}\right.
\end{equation*}
with notation analogous to that of \eqref{esti1}.\\
We consider the continuous function $\mathcal{F}$ with respect to each of its variables defined by:
\begin{align}
\mathcal{F}(\alpha,v)=v_K-\alpha\biggl[u_K^0-\frac{\delta t}{m(K)}\displaystyle\sum_{\sigma\in\bar{\mathlarger\varepsilon}_K}F_{K,\sigma}(v_{K}\!,v_{K,\sigma})+\frac{\delta t}{m(K)}\displaystyle\sum_{\sigma\in\bar{\mathlarger\varepsilon}_K}\tau_{K,\sigma}\biggl(\phi(v_{K,\sigma})-\phi(v_{K})\biggr)\biggr].
\end{align}
The function $\mathcal{F}(\alpha,.)$ is a continuous homotopy between $\mathcal{F}(0,.)$ and $\mathcal{F}(1,.)$. First, remark that $u^{n+1}_K=0$ is solution of $\mathcal{F}(0,u^{n+1}_K)=0$ for all $(n,K)\in[0,N]\times\mathcal{O}$. If $B$ is a ball with a sufficiently large radius in the space of solution of the system, the equation $\mathcal{F}(.,.)=0$ has no solution on the boundary  $\partial B$. Indeed replacing $u_0$, $f$, $\phi$ by $\alpha u_0$, $\alpha f$, $\alpha\phi$ we can apply the argument of Proposition \ref{exi} to solutions of equation $\mathcal{F}(\alpha,v)=0$.  Then it is enough to supply the finite dimensional set $\mathbb R^\theta$ of discrete functions by the norm $||\cdot||_{L^\infty}$ and take $B$ of radius larger than $u_{\max}$. Therefore%,% we can apply  Theorem \ref{Schauder}. We get:
\begin{align}
degree(\mathcal{F}(0,.),B)=degree(\mathcal{F}(1,.),B)\neq 0.
\end{align}
Thus there exists at least a solution to equation $\mathcal{F}(1,.)=0$. This solution is a solution to our scheme.  
\end{proof}
%%%%%%%%%%%%%%%%%%%%%%%%%%%%%%%%%%%%%%%%%%%%%%%%%%%%%%%%%%%%%%%%%%%%%%%%%%%%%%%%%%%%%%%%%%%%%%%%%%%%%%%%%%%%%%%%%%%%%%%%%%%%%%%%%%%%%%%%%%%%%%%%%%%%%%%%%%%%%%%%%%%%
\section{Continuous entropy inequality}
We prove in this section that the approximate solutions fulfill a continuous entropy inequality in the sense of Theorem \ref{continousentropy} below. Before, we recall a result that will serve us in the proof of this Theorem. 
\begin{lem}\label{magique}( see e.g. \normalfont{J ~.Droniou \cite{Droniou}})
Let $K$ be a non empty open convex polygonal set in $\mathbb R^\ell$. For $\sigma\in\bar{\mathlarger\varepsilon}_K$, we denote by $x_\sigma$ the center of gravity of $\sigma$; we also denote by $n_{K,\sigma}$ the unit normal vector to $\sigma$ outward to $K$. Then, for all vector $\vec{V}\in\mathbb R^\ell$ and for all point $x_K\in K$, we have:
\begin{align}\label{magi0}
m(K)\vec{V}=\displaystyle\sum_{\sigma\in\bar{\mathlarger\varepsilon}_K}m(\sigma)\vec{V}.n_{K,\sigma}(x_\sigma-x_K).
\end{align}
\end{lem}
\begin{proof}
We denote by a superscript $i$, the $i-$ th coordinate of vectors and points in $\mathbb R^\ell$. By Stokes formula, we have:
\begin{align}\label{magi1}
m(K)V^i=\displaystyle\int_K\div((x^i-x_K^i)\vec{V})dx&=\displaystyle\int_{\partial K}(x^i-x^i_K)\vec{V}.n_{K}d\gamma(x)\nonumber\\
&=\displaystyle\sum_{\sigma\in\bar{\mathlarger\varepsilon}_K}\displaystyle\int_\sigma(x^i-x^i_K)\vec{V}.n_{K,\sigma}d\gamma(x).
\end{align}
Hence, by the definition of the center of gravity, we have:
\begin{align}\label{magi2}
\displaystyle\int_\sigma(x^i-x^i_K)d\gamma(x)=\displaystyle\int_\sigma x^id\gamma(x)-m(\sigma)x_K^i=m(\sigma)x^i_\sigma-m(\sigma)x^i_K.
\end{align}
Remplace \eqref{magi2} in \eqref{magi1}; we find \eqref{magi0}.
\end{proof}
From now on, as the approximate solutions satisfy the discrete entropy inequalities \eqref{entropydiscrete01}, we prove that its satisfies  a continuous form of these inequalities.
\begin{thm}\label{continousentropy}
Assume that \eqref{reg1}, \eqref{monotony}-\eqref{consistency}  hold. Let $\uOdt$ be the approximate solution of the problem $(P)$ defined by \eqref{esti0},\eqref{esti1}. Then the following continuous approximate entropy inequalities hold: for all $k\in[0,u_{\max}]$, for all $\xi\in\mathcal{C}^\infty([0,T)\times\mathbb R^\ell)$, $\xi\geq0$,
\begin{align}\label{content}
&\displaystyle\int_0^T\!\!\!\int_\Omega\displaystyle \left\{\eta_k(\uOdt)\xi_t+\biggl(\Phi_k(\uOdt)-\nabla_{\mathcal{O}}\eta_{\phi(k)}(\phi(\uOdt))\biggr).\nabla\xi\displaystyle\right\}dxdt\nonumber\\&+\displaystyle\int_\Omega \eta_k(u_0)\xi(0,x)dx+\displaystyle\int_0^T\!\!\!\int_{\partial\Omega} \left|f(k).\eta(x)\right|\xi(t,x) d{\mathcal{H}}^{\ell-1}(x)dt\geq -\upsilon_{\mathcal{O},n}(\xi);
\end{align}
where: $\forall \xi\in\mathcal{C}^\infty([0,T)\times\mathbb R^\ell)$, $\upsilon_{\mathcal{O},n}(\xi)\rightarrow0$ when $h\rightarrow0$. Here
\begin{align*}%\label{grad5}
\nabla_{\mathcal{O}}\eta_{\phi(k)}(\phi(\uOdt))=\displaystyle\sum_{n=0}^N1_{[t_n,t_{n+1}]}\displaystyle \sum_{K|L}1_{\widehat{K|L}}\nabla_{\widehat{K|L}}\eta_{\phi(k)}(\phi(\uOdt)).
\end{align*}
\end{thm}
\begin{rem}
In the same case, if we replace in \eqref{content} $\eta_k$ by $\eta^+_k$ (resp $\eta^-_k$)  and $\left|f(k).\eta(x)\right|$ by $\left(f(k).\eta(x)\right)^+$ (resp $\left(f(k).\eta(x)\right)^-$) we obtain sub entropy inequalities (resp super entropy inequalities). Obviously, the approximate solution $\uOdt$ is an approximate entropy solution if and only if $\uOdt$ is aproximate entropy sub-solution and entropy super-solution simultaneously. 
\end{rem}
\begin{proof}[Proof of Theorem~\ref{continousentropy}]
Let $\xi\in\mathcal{C}^\infty([0,T)\times\mathbb R^l)^+$ and $k\in[0,u_{\max}]$, we fix $T\geq0$ and set $N=\frac{T}{\delta t}+1$. 
It  is enough to suppose that $\xi(t,x)=\theta(t)\zeta(x)$, this mean that $\xi^{n+1}_K=\theta^{n+1}\zeta_K$. By density in $\mathcal{C}^\infty([0,T[\times \mathbb R^\ell)$ of linear combinations of such functions, the general case will follow. Depending on the circumstances, $\zeta_K=\fint_K\zeta$ or $\zeta_K=\zeta(x_K)$ with $x_K$ the center of control volume $K$. \\
Multiplying inequality \eqref{entropydiscrete01} by $\delta t\xi^{n+1}_K$ and summing over $K\in\mathcal{O}$ and $n\in\{0,...,N\}$, yields the inequality $I^{Disc}_{Evol}+I^{Disc}_{Conv}+I^{Disc}_{Diff}\leq 0$,
 where:
\begin{align}
I^{Disc}_{Evol}&=\displaystyle\sum_{n=0}^N\displaystyle\sum_{K\in\mathcal{O}}m(K)\biggl(\eta_k(u_K^{n+1})-\eta_k(u_K^{n})\biggr)\xi^{n+1}_K,\\
I^{Disc}_{Conv}&=\displaystyle\sum_{n=0}^N\delta t\displaystyle\sum_{K\in\mathcal{O}}\displaystyle\sum_{\sigma\in\mathlarger\varepsilon_K}\Phi_{K,\sigma,k}(u_K^{n+1},u_{K,\sigma}^{n+1}) \xi_K^{n+1}\nonumber\\-&\displaystyle\sum_{n=0}^N\delta t\displaystyle\sum_{K\in\mathcal{O}}\displaystyle\sum_{\sigma\in\mathlarger\varepsilon_K^{ext}}sign(u_K^{n+1}-k)F_{K,\sigma,k}(k,k) \xi_K^{n+1},\\
I^{Disc}_{Diff}&=-\displaystyle\sum_{n=0}^N\delta t\displaystyle\sum_{K\in\mathcal{O}}\displaystyle\sum_{K|L}\tau_{K|L}\biggl(\eta_{\phi(k)}(\phi(u_{L}^{n+1}))-\eta_{\phi(k)}(\phi(u_{K}^{n+1}))\biggr)\xi_K^{n+1}.
\end{align}
To prove  inequality \eqref{content}, we have to prove that $I^{Cont}_{Evol}+I^{Cont}_{Conv}+I^{Cont}_{Diff}\leq\upsilon_{\mathcal{O},n}(\xi)$ where $I^{Cont}_{Evol}$, $I^{Cont}_{Conv}$ and $I^{Cont}_{Diff}$ are defined by:
\begin{align*}\label{content1}
&I^{Cont}_{Evol}=-\displaystyle\int_0^T\int_\Omega\displaystyle\eta_k(\uOdt)\zeta(x)\theta_t(t)dxdt-\displaystyle\int_\Omega \eta_k(u_0)\theta(0)\zeta(x)dx,\\
&I^{Cont}_{Conv}=-\displaystyle\int_0^T\theta\int_\Omega\displaystyle\Phi_k(\uOdt).\nabla\zeta dxdt-\displaystyle\int_0^T\theta\int_{\partial\Omega} \left|f(k).\eta(x)\right|\zeta(x) d{\mathcal{H}}^{\ell-1}(x)dt,\\
&I^{Cont}_{Diff}=\displaystyle\int_0^T\theta\int_\Omega\displaystyle\nabla_{\mathcal{O}}\eta_{\phi(k)}(\phi(\uOdt)).\nabla\zeta dxdt.
\end{align*}
Then, we have to compare $I^{Disc}_{Evol}$ with $I^{Cont}_{Evol}$; $I^{Dicr}_{Conv}$ with $I^{Cont}_{Conv}$; and $I^{Disc}_{Diff}$ with $I^{Cont}_{Diff}$.\\
Firstly,  we have to  estimate $|I^{Disc}_{Evol}-I^{Cont}_{Evol}|$. Using the definition of  $\uOdt$, the quantity $I^{Disc}_{Evol}$ reads:
\begin{align}
I^{Disc}_{Evol}&=-\displaystyle\sum_{n=0}^{N-1}\displaystyle\sum_{K\in\mathcal{O}}m(K)\eta_k(u_K^{n+1})\biggl(\xi^{n+1}_K-\xi^{n}_K\biggr)-\displaystyle\sum_{K\in\mathcal{O}}m(K)\biggl(\eta_k(u_K^{0})\xi^{1}_K-\eta_k(u_K^{N+1})\xi^{N+1}_K\biggr)\nonumber\\
&=-\displaystyle\sum_{n=0}^{N-1}\delta t\displaystyle\sum_{K\in\mathcal{O}}m(K)\eta_k(u_K^{n+1})\frac{\xi^{n+1}_K-\xi^{n}_K}{\delta t}-\displaystyle\sum_{K\in\mathcal{O}}m(K)\eta_k(u_K^{0})\xi^{1}_K\nonumber\\
&=-\displaystyle\sum_{n=0}^{N-1}\delta t\displaystyle\sum_{K\in\mathcal{O}}m(K)\eta_k(u_K^{n+1})\frac{\theta^{n+1}-\theta^{n}}{\delta t}\zeta_K-\displaystyle\sum_{K\in\mathcal{O}}m(K)\eta_k(u_K^{0})\theta^{1}\zeta_K\nonumber\\
&=-\displaystyle\sum_{n=0}^{N-1}\delta t\displaystyle\sum_{K\in\mathcal{O}}m(K)\eta_k(u_K^{n+1})(\theta^{n})_t\fint_K\zeta(x)dx-\displaystyle\sum_{K\in\mathcal{O}}m(K)\eta_k(u_K^{0})\theta^{1}\fint_K\zeta(x)dx,
\end{align}
with $(\theta^{n})_t=\displaystyle\int_{t_n}^{t_{n+1}}\theta_t dt$.
We deduce that
\begin{align}
|I^{Disc}_{Evol}-I^{Cont}_{Evol}|\leq\upsilon_{\mathcal{O},k}^1(\xi)+\upsilon_{\mathcal{O},k}^2(\xi),
\end{align}
where:
\begin{align}\label{erreur1}
&\upsilon_{\mathcal{O},k}^1(\xi)=\displaystyle\sum_{n=1}^{N-1}\delta t\displaystyle\sum_{K\in\mathcal{O}}m(K)\eta_k(\uOdt)\mathbf 1_K\mathbf 1_{[t_n,t_{n+1}]}(\theta^{n})_t\biggl|\fint_K\zeta(x)dx-\zeta_K\biggr|,\\
&\upsilon_{\mathcal{O},k}^2(\xi)=\displaystyle\sum_{K\in\mathcal{O}}m(K)\eta_k(u_K^{0})\theta^{1}\biggl|\fint_K\zeta(x)dx-\zeta_K\biggr|.
\end{align}
%%%%%%%%%%%%%%%%%%%%%%%%%
As $\xi\in\mathcal{C}^\infty$, then we have:
\begin{align}
\biggl|\fint_K\zeta(x)dx-\zeta_K\biggr|\leq||\zeta||_{\mathcal{C}^1}h.
\end{align}
Then, the quantities $\upsilon_{\mathcal{O},k}^1(\xi)$, $\upsilon_{\mathcal{O},k}^2(\xi)$,  tend to zero when $h\rightarrow 0$. \\
Secondly, we study the difference between $I^{Disc}_{Conv}$ and $I^{Cont}_{Conv}$. We take care separately of  what happens inside and what happens on the boundary of $\Omega$. Therefore we write $I^{Cont}_{Conv}$ has the sum of $I^{Cont,int}_{Conv}$, and $I^{Cont,ext}_{Conv}$.%In fact%  with
\begin{align}\label{content1}
&I^{Cont,int}_{Conv}=-\displaystyle\int_0^T\theta\int_\Omega\displaystyle\Phi_k(\uOdt).\nabla\zeta dxdt\nonumber\\
&I^{Cont,ext}_{Conv}=-\displaystyle\int_0^T\theta\int_{\partial\Omega} \left|f(k).\eta(x)\right|\zeta(x) d{\mathcal{H}}^{\ell-1}(x)dt.
\end{align}
Further, introduce auxiliary values $(\zeta_{K|L})_{L\in\mathcal{N}(K)}$ by $\zeta_{K|L}=\zeta(x_{K|L})$, where $x_{K|L}$ is the barycenter of $K|L$.
The term $I^{Disc}_{Conv}$, which can be rewritten as the sum between $I^{Disc,int}_{Conv}$ and $I^{Disc,ext}_{Conv}$:
\begin{align}
I^{Disc,int}_{Conv}&=\displaystyle\sum_{n=0}^N\delta t\displaystyle\sum_{\sigma\in K|L}\Phi_{K,K|L,k}(u_K^{n+1},u_{L}^{n+1})\biggl[(\xi_K^{n+1}-\xi_{K|L}^{n+1})-(\xi_L^{n+1}-\xi_{K|L}^{n+1})\biggr]\nonumber\\
&=\displaystyle\sum_{n=0}^N\delta t\displaystyle\sum_{K\in\mathcal{O}}\displaystyle\sum_{L\in\mathcal{N}(K)}\Phi_{K,K|L,k}(u_K^{n+1},u_{L}^{n+1})(\xi_K^{n+1}-\xi_{K|L}^{n+1})\nonumber\\
&=\displaystyle\sum_{n=0}^N\theta^{n+1}\delta t\displaystyle\sum_{K\in\mathcal{O}}\displaystyle\sum_{L\in\mathcal{N}(K)}\Phi_{K,K|L,k}(u_K^{n+1},u_{L}^{n+1})(\zeta_K-\zeta_{K|L})\nonumber\\
&=-\displaystyle\sum_{n=0}^N\theta^{n+1}\delta t\displaystyle\sum_{K\in\mathcal{O}}\displaystyle\sum_{L\in\mathcal{N}(K)}\Phi_{K,K|L,k}(u_K^{n+1},u_{L}^{n+1})(\zeta_{K|L}-\zeta_K)\\
I^{Disc,ext}_{Conv}&=-\displaystyle\sum_{n=0}^N\delta t\displaystyle\sum_{K\in\mathcal{O}}\displaystyle\sum_{\sigma\in\mathlarger\varepsilon_K^{ext}}sign(u_K^{n+1}-k)F_{K,\sigma}(k,k) \xi_K^{n+1}\nonumber\\
&=-\displaystyle\sum_{n=0}^N\theta^{n+1}\delta t\displaystyle\sum_{K\in\mathcal{O}}\displaystyle\sum_{\sigma\in\mathlarger\varepsilon_K^{ext}}sign(u_K^{n+1}-k)F_{K,\sigma}(k,k) \zeta_K.
\end{align}
Now, we compare $I^{Cont,int}_{Conv}$ and $I^{Disc,int}_{Conv}$. As the numerical fluxes, the numerical entropy fluxes are consistent:
\begin{align*}
\Phi_{K,\sigma,k}(u_K^{n+1},u_K^{n+1})=\int_\sigma\Phi_{k}(u_K^{n+1}).n_{K,\sigma}d\gamma(x)dt=m(K|L)\Phi_{k}(u_K^{n+1}).n_{K|L}.
\end{align*}
Simultaneously, for each $K\in\mathcal{O}$, we approach $\zeta$ by the affine function $\tilde{\zeta}_K$ in a neighborhood of $K$,  with $\tilde{\zeta}_{K}=\tilde{\zeta}(x_K)$, we set $\tilde{\zeta}_{K|L}=\tilde{\zeta}(x_{K|L})$. Then
\begin{align}\label{affinefunction}%\tag{af}
%\left\{\begin{array}{rll}
&\zeta(x)1_K=\tilde{\zeta}_K+\underline{0}(|x-x_K|^2); \;\;\zeta_{K|L}-\tilde{\zeta}_{K|L}=\underline{0}(h^2); \;\;\nabla\tilde{\zeta}_K=\mbox{ cst  on } K \nonumber\\&||\nabla\zeta-\nabla\tilde{\zeta}||_{L^\infty(K)}=\underline{0}(h) \mbox{ and } \nabla\tilde{\zeta}_K.(x_K-x_{K|L})=\tilde{\zeta}_K-\tilde{\zeta}_{K|L}.
%\end{array}\right.
\end{align}
 We denote the resulting expression by $\tilde{I}^{Disc,int}_{Conv}$, we have
\begin{align*}
\tilde{I}^{Disc,int}_{Conv}&=-\displaystyle\sum_{n=0}^N\delta t\displaystyle\sum_{K\in\mathcal{O}}\displaystyle\sum_{L\in \mathcal{N}(K)}\Phi_{K,K|L,k}(u_K^{n+1},u_K^{n+1})(\tilde{\xi}^{n+1}_{K|L}-\tilde{\xi}^{n+1}_K)\\
&=-\displaystyle\sum_{n=0}^N\delta t\displaystyle\sum_{K\in\mathcal{O}}\displaystyle\sum_{L\in \mathcal{N}(K)}m(K|L)\Phi_k(u^{n+1}_K).n_{K|L}(\tilde{\xi}^{n+1}_{K|L}-\tilde{\xi}^{n+1}_K)\\
&=-\displaystyle\sum_{n=0}^N\delta t\theta^{n+1}\displaystyle\sum_{K\in\mathcal{O}}\displaystyle\sum_{L\in \mathcal{N}(K)}m(K|L)\Phi_k(u^{n+1}_K).n_{K|L}(\tilde{\zeta}_{K|L}-\tilde{\zeta}_K)\\
&=-\displaystyle\sum_{n=0}^N\delta t\theta^{n+1}\displaystyle\sum_{K\in\mathcal{O}}\displaystyle\sum_{L\in \mathcal{N}(K)}m(K|L)\Phi_k(u^{n+1}_K).n_{K|L}\nabla\tilde{\zeta}_K.(x_{K|L}-x_K)\\
&=-\displaystyle\sum_{n=0}^N\delta t\theta^{n+1}\displaystyle\sum_{K\in\mathcal{O}}\Phi_k(u^{n+1}_K).\displaystyle\sum_{L\in \mathcal{N}(K)}m(K|L)\nabla\tilde{\zeta}_K.(x_{K|L}-x_K)n_{K|L}.\\
\end{align*}
From now, using Lemma \ref{magique}, which states that%
\begin{align}
\displaystyle\sum_{K|L}m(K|L)\nabla\tilde{\zeta}_K.n_{K|L}(x_{K|L}-x_K)=m(K)\nabla\tilde{\zeta}_K,
\end{align}
we find:
\begin{align*}
\tilde{I}^{Disc,int}_{Conv}&=-\displaystyle\sum_{n=0}^N\delta t\theta^{n+1}\displaystyle\sum_{K\in\mathcal{O}}\Phi_k(u^{n+1}_K)m(K)\nabla\tilde{\zeta}_K.
\end{align*}
It is easy to see that
\begin{align*}
\tilde{I}^{Disc,int}_{Conv}=-\displaystyle\int_0^T\theta\int_\Omega\displaystyle\Phi_k(\uOdt).\nabla\tilde{\zeta}_K dxdt=:\bar{I}^{Disc,int}_{Conv}
\end{align*}
\begin{align}
|I^{Disc,int}_{Conv}-I^{Cont,int}_{Conv}|\leq&|I^{Disc,int}_{Conv}-\tilde{I}^{Disc,int}_{Conv}|+|\bar{I}^{Disc,int}_{Conv}-I^{Cont,int}_{Conv}|\nonumber\\
&=\upsilon_{\mathcal{O},k}^3(\xi)+\upsilon_{\mathcal{O},k}^4(\xi)
\end{align}
 with:
 \begin{align}\label{erreur2}
\upsilon_{\mathcal{O},k}^3(\xi)=&\displaystyle\sum_{n=0}^N\delta t\theta^{n+1}\!\displaystyle\sum_{K|L}\biggl|\biggl(\Phi_{K,L,k}(u_{K}^{n+1}\!,u_{K}^{n+1})-\Phi_{K,L,k}(u_{K}^{n+1}\!,u_{L}^{n+1})\biggr)\biggl(\zeta_K-\zeta_{K|L}\biggr)\biggr|;\nonumber\\
\upsilon_{\mathcal{O},k}^4(\xi)=&\displaystyle\int_0^T\theta\int_\Omega\displaystyle\Phi_k(\uOdt).|\nabla\zeta- \nabla\tilde{\zeta}_K|dxdt.
\end{align}
Let us show that $\upsilon_{\mathcal{O},k}^3(\xi)$ and $\upsilon_{\mathcal{O},k}^4(\xi)$  tend to zero as $h\rightarrow0$. Thanks to \eqref{affinefunction}, $\upsilon_{\mathcal{O},k}^4(\xi)$ as $h\rightarrow0$.
Now, we write:
\begin{align}\label{err1}
\xi(t,x)-\hat{\xi}_{K|L}^{n+1}=\frac{1}{\delta tm(K|L)}\int_{n\delta t}^{(n+1)\delta t}\!\!\!\!\int_{K|L}(\xi(t,x)-\xi(s,y))d\gamma(y)ds.
\end{align}
For all $(x,y)\in K|L\times K|L$, 
\begin{align}\label{err2} 
|\zeta(x)-\zeta(y)|\leq h||\nabla\zeta||_{L^\infty}.
\end{align}
We exploit the BV-weak estimates on space derivatives to prove that $\upsilon_{\mathcal{O},k}^3(\xi)$ tend to zero when $h$ goes to zero. Indeed, we have
\begin{align*}
|\Phi_{K,L,k}(u_{K}^{n+1}\!,u_{K}^{n+1})-\Phi_{K,L,k}(u_{K}^{n+1}\!,u_{L}^{n+1})|\leq \max_{u_L^{n+1}\leq c\leq d\leq u^{n+1}_K}(F_{K,\sigma}(d,c)-(F_{K,\sigma}(d,d))
\end{align*}
and thanks to \eqref{err2}, we get an estimate on the difference between the average value of $\zeta$ and a control volume and on one of its edges: there exists $C_{\zeta}$ depending only upon $\zeta$, such that
\begin{align*}
\forall K|L,\;\; |\zeta_K-\zeta_{K|L}^{n+1}|\leq C_\zeta h.
\end{align*} 
Therefore, the following estimate on $\upsilon_{\mathcal{O},k}^3(\xi)$ holds:
\begin{align*}
&\upsilon_{\mathcal{O},k}^3(\zeta)=C_\zeta(h)\displaystyle\sum_{n=0}^N\delta t\displaystyle\sum_{K|L}\biggl[\max_{u_L^{n+1}\leq c\leq d\leq u_K^{n+1}}(F_{K,\sigma}(d,c)-F_{K,\sigma}(d,d))\biggr]\nonumber\\&
	+C_\zeta(h)\displaystyle\sum_{n=0}^N\delta t\displaystyle\sum_{K|L}\biggl[\max_{u_L^{n+1}\leq c\leq d\leq u_K^{n+1}}(F_{K,\sigma}(d,c)-F_{K,\sigma}(c,c))\biggr]\nonumber\\&\leq C_\zeta \frac{h}{\sqrt{h}}
\end{align*}
where the constant $C_\xi$ is given by \eqref{bvestimate}.
Now, it remains to notice that %compare $I^{Cont,ext}_{Conv}$ to $I^{Disc,ext}_{Conv}$. 
\begin{align}
-I^{Disc,ext}_{Conv}&=\displaystyle\sum_{n=0}^N\theta^{n+1}\delta t\displaystyle\sum_{K\in\mathcal{O}}\displaystyle\sum_{\sigma\in\mathlarger\varepsilon_K^{ext}}sign(u_K^{n+1}-k)F_{K,\sigma}(k,k) \zeta_K\nonumber\\&\leq\displaystyle\sum_{n=0}^N\theta^{n+1}\delta t\displaystyle\sum_{K\in\mathcal{O}}\displaystyle\sum_{\sigma\in\mathlarger\varepsilon_K^{ext}}|F_{K,\sigma}(k,k)| \zeta_K=-I^{Cont,ext}_{Conv}.
\end{align}
Then, we have:
\begin{align}
I^{Cont,ext}_{Conv}-I^{Disc,ext}_{Conv}\leq0.
\end{align}
 The last step is to compare $I^{Cont,ext}_{Diff}$ to $I^{Disc,ext}_{Diff}$. We rewrite the term $I^{Disc}_{Diff}$ as
\begin{align}
I^{Disc}_{Diff}&=-\displaystyle\sum_{n=0}^N\delta t\displaystyle\sum_{\sigma\in {K|L}}\tau_{K|L}\biggl(\eta_{\phi(k)}(\phi(u_{L}^{n+1}))-\eta_{\phi(k)}(\phi(u_{K}^{n+1}))\biggr)\biggl(\xi_K^{n+1}-\xi_L^{n+1}\biggr)\nonumber\\
&=-\displaystyle\sum_{n=0}^N\delta t\theta^{n+1}\displaystyle\sum_{\sigma\in K|L}\frac{m(K|L)}{d_{K,L}}\biggl(\eta_{\phi(k)}(\phi(u_{L}^{n+1}))-\eta_{\phi(k)}(\phi(u_{K}^{n+1}))\biggr)\biggl(\zeta_K-\zeta_L\biggr).\nonumber\\
&=-\displaystyle\sum_{n=0}^N\delta t\theta^{n+1}\displaystyle\sum_{\sigma\in K|L}\frac{m(K|L)d_{K,L}}{\ell}\biggl(\ell\nabla_{\widehat{K|L}}\eta_{\phi(k)}(\phi(u_{\mathcal{O}}^{n+1}))\biggr).\biggl(\frac{\zeta_K-\zeta_L}{d_{K,L}}.n_{K|L}\biggr)\nonumber\\
&=-\displaystyle\sum_{n=0}^N\delta t\theta^{n+1}\displaystyle\sum_{\sigma\in K|L}m(\widehat{K|L})\nabla_{\widehat{K|L}}\eta_{\phi(k)}(\phi(u_{\mathcal{O}}^{n+1}))\tilde{\nabla}_{\widehat{K|L}}\zeta
\end{align}
where: $\tilde{\nabla}_{\widehat{K|L}}\zeta=\fint_{x_K}^{x_L} \nabla\zeta$.
Notice that 
\begin{align*}
||\nabla\zeta-\tilde{\nabla}_{\widehat{K|L}}\zeta||_{L^\infty(\widehat{K|L})}=\underline{o}(h).
\end{align*}
Therefore we have
\begin{align}
|I^{Disc}_{Diff}-I^{Cont}_{Diff}|&\leq\upsilon_{\mathcal{O},k}^5(\xi),
\end{align}
with:
 \begin{align}\label{erreur5}
\upsilon_{\mathcal{O},k}^5(\xi)=&\displaystyle\int_0^T\theta\int_\Omega\displaystyle|\nabla_{\mathcal{O}}\eta_{\phi(k)}(\phi(\uOdt))|.|\nabla\zeta-\tilde{\nabla}_{\widehat{K|L}}\zeta| dxdt.
\end{align}
To conclude, we prove that $\upsilon_{\mathcal{O},k}^5(\xi)\rightarrow0$ as $h\rightarrow0$. Using Cauchy-Schwarz inequality, we find 
\begin{align*}
\upsilon_{\mathcal{O},k}^5(\xi)\leq||\theta||_{L^\infty}||\nabla_{\mathcal{O}}\eta_{\phi(k)}(\phi(\uOdt))||_{L^2}\underline{0}(h).\end{align*}
Then, using the fact that $\eta$ is $1-$Lipschitz, and the estimate \eqref{H1estima}  we prove that $\upsilon_{\mathcal{O},k}^5(\xi)\rightarrow0$ as $h\rightarrow0$.
\end{proof}
%%%%%%%%%%%%%%%%%%%%%%%%%%%%%%%%%%%%%%%%%%%%%%%%%%%%%%%%%%%%%%%%%%%%%%%%%%%%%%%%%%%%%%%%%%%%%%%%%%%%%%%%%%%%%%%%%%%%%%%%%%%%%%%%%%%%%%%%%%%%%%%%%%%%%%%%%%%%%%%%%%%%%
\begin{section}{Convergence of the scheme}
The main result of this paper is the following theorem.
\begin{thm}\label{conv}(Convergence of the approximate solution  towards the entropy solution).
Assume that one of the following hypotheses is satisfied
\begin{align}\label{dimension}\tag{$H_{\ell=1}$}
\ell=1 \mbox{ and } \Omega=(a,b) \mbox{ an interval of }\mathbb R;
\end{align}
\begin{align}\label{conv-uc}\tag{$H_{u_c=0}$}
\ell\geq1\;\;u_c=0,\; \mbox{ and } f\circ\phi^{-1}\in\mathcal{C}^{0,\alpha}, \alpha>0;
\end{align}
\begin{align}\label{conv-umax}\tag{$H_{u_c=u_{\max}}$}
\ell\geq1\;\;u_c=u_{\max}.% \mbox{ and } \Omega=(a,b) \mbox{ an interval of }\mathbb R.
\end{align}
 Let, $(\uOdt)_{\mathcal{O},\delta t}$ be a family of approximate solutions of problem $(P)$ defined by  \eqref{esti0}, \eqref{esti1}. Then, under hypotheses \eqref{reg2}-\eqref{consistency}, we have:
\begin{align}
\forall p\in[1,+\infty)\;\;\uOdt\longrightarrow u \mbox{ in } L^p(Q) \mbox{ as } \max(\delta t, h)\longrightarrow0;
\end{align}
\begin{align*}
\nabla_{\mathcal{O}}\phi(u_{\mathcal{O},n})\rightharpoonup\nabla\phi(u)\mbox{ in } L^2(Q) \mbox{ as } \max(\delta t, h)\longrightarrow0
\end{align*}
where $u$ is the unique entropy solution of $(P)$, i.e $u$ satisfies \eqref{ESzeroflux}.
\end{thm}
\begin{rem}\label{interessant}
It is possible to replace in the Theorem \ref{conv} all the three hypotheses \eqref{dimension}, \eqref{conv-uc}, \eqref{conv-umax} by the following one, which is much more general:
\begin{align}\label{dimensionsup}\tag{$H_{\mbox{reg}}(u_0)$}
\left \{\begin{array}{ll}
&\ell\geq 1 \mbox{ and } u_0 \mbox{ is such that there exist an entropy solution u of } (P) \mbox{  such that  }\nonumber\\& (f(u)-\nabla\phi(u)).\eta(x) \mbox{ possess a strong trace in }L^1 \mbox{ sense}.
\end{array}\right.
\end{align}
Such kind of function $u$ satisfying \eqref{dimensionsup}, will be called trace regular entropy solution (see \cite{Boris}). The idea to prove uniqueness of entropy solution is to compare any entropy solution of $(P)$ with trace regular entropy solution and break the symmetry in the application of doubling of variables method by taking test function that is zero on the boundary $Q\times ((0,T)\times\partial\Omega)$ of $Q\times Q$ but non zero on the boundary $((0,T)\times\partial\Omega)\times Q$ (see the method of \cite{BF,BG}). If \eqref{dimensionsup} is satisfied for all $u_0$ that belong  to a certain subset $X$ such that $\overline{X}^{||.||_{L^1}}=L^1(\Omega;[0,u_{\max}])$, then uniqueness is true for all $u_0$.\\
 Presently to our knowledge the only  results which establish that \eqref{dimensionsup} hold for a dense subset $X$ is proved for the case  \eqref{conv-umax} (see \cite{BFK1,PAN}).\\ In this pure hyperbolic case existence of the strong trace of the flux is established in  \cite{BFK1,PAN}. Then uniqueness of entropy solution follows by standard doubling of variable methods and it is enough to take a  symmetric test function.\\
In the case where hypotheses \eqref{conv-uc} or \eqref{dimension} are satisfied, it is more easy to prove existence of trace regular entropy solution for the stationary problem with $L^\infty$ source term. In this case, we even have sense that  the total flux is continuous up to the boundary, i.e $(f(u)-\nabla\phi(u)).\eta\in\mathcal{C}(\overline{\Omega})$ (see \cite{Liberman}), \cite{BT}). Then we can adopt the same strategy as in the case where  \eqref{dimensionsup} hold , but in the doubling of variable method we compare entropy solution of $(P)$ with trace regular entropy solution of $(S$). Then using nonlinear semigroup approach, we prove that entropy solution of $(P)$ is the unique mild solution (see \cite{ BG,BF}). The same strategy is adopted here to prove that entropy-process solution (see Definition \ref{entrprosol}) is the unique entropy solution ( see Appendix $1$ and $2$). 
\end{rem}
\begin{proof}[Proof of Theorem~\ref{conv}]
The proof of Theorem \ref{conv} is in two steps. First in Proposition \ref{conventropyprocess} , we prove that the approximate solutions converge towards an entropy-process solution. Then in Appendix $2$ (see Theorem \ref{unicite}, and Proposition \ref{uc-umax}, \ref{prouc}, \ref{proumax}) we prove that entropy-process solution is in fact the unique entropy solution using the intermediate notion of integral-process solution developed for this purpose in the Appendix $1$. 
\end{proof}
\begin{subsection}{Entropy process solution}
\begin{defn}\label{entrprosol}
Let $\mu\in L^\infty(Q\times(0,1))$. The function $\mu=\mu(t,x,\alpha)$ taking values in $[0,u_{\max}]$ is called an entropy-process solution to  problem $(P)$ if $\forall k\in  [0,u_{\max}]$,  $\forall\xi\in \mathcal{C}^\infty([0,T)\times\mathbb R^\ell)$, with $\xi\geq 0$, the following inequality   holds :
\begin{align}\label{ESP1}
&\displaystyle\int_0^T\int_\Omega\int_0^1\displaystyle \left\{|\mu(\alpha)u-k|\xi_t+sign(\mu(\alpha)-k)\Bigl[f(\mu)-f(k)\Bigr].\nabla\xi\displaystyle\right\}dxdtd\alpha\nonumber\\&-\displaystyle\int_0^T\int_\Omega\nabla|\phi(u)-\phi(k)|.\nabla\xi dxdt+\displaystyle\int_0^T\int_{\partial\Omega} \left|f(k).\eta(x)\right|\xi(t,x) d{\mathcal{H}}^{\ell-1}(x)dt\nonumber\\&+\displaystyle\int_\Omega |u_0-k|\xi(0,x)dx\geq 0,
\end{align}
where $u(t,x)=\displaystyle\int_0^1\mu(t,x,\alpha)d\alpha$.
\end{defn}
\begin{rem}\label{entrprosol1}
If $\mu\in L^\infty(Q\times(0,1))$ is entropy process solution then, it satisfies for all $\xi\in L^2(0,T;H^1(\Omega))$ such that $\xi_t\in L^1(Q)$ and $\xi(T,.)=0$ 
\begin{equation}\label{entropyweakproc}
\displaystyle\int_0^1\int_0^T\int_\Omega\displaystyle\left\{\mu\xi_t+\biggl(f(\mu)-\nabla\phi(u)\biggr).\nabla\xi\right\}dxdtd\alpha+\int_\Omega u_0\xi(0,x)dx=0.
\end{equation} 
\end{rem}
We recall the nonlinear weak star convergence for $(\uOdt)_{\mathcal{O},\delta t}$ which is equivalent to the notion of convergence towards  a Young measure as developed in \cite{Diperna}.
\begin{thm}\label{weakstar}\normalfont{( R. Eymard, T. Gallou\"{e}t, and R. Herbin, \cite{EGH})} (Nonlinear weak star Convergence)
Let $(u_n)_{n\in\mathbb N}$ be a bounded sequence in $L^\infty(Q)$. Then, there exists $\mu\in L^\infty(Q\times(0,1))$, such that up to a subsequence, $u_n$  tends to $\mu$ in the nonlinear weak star sense as $n\longrightarrow\infty$, i.e:
\begin{align}
\forall h\in\mathcal{C}(\mathbb R,\mathbb R),\;\; h(u_n){\rightharpoonup}\displaystyle\int_0^1 h(\mu(.,\alpha))d\alpha\mbox{ weakly }-*\mbox{ in }  L^\infty(Q) 
\end{align}
%where $\stackrel{nl-*}{\rightharpoonup}$ mean the convergence for weak star topology in  $L^\infty(Q)$. 
Moreover,  if $\mu$ is independent on $\alpha$ (i.e $\mu(t,x,\alpha)=u(t,x)$ for a.e. $(t,x)$, and for all $\alpha$), then $u_n$ converge strongly in $L^1(Q)$ towards some $u(t,x)$. In particular, observe that the following holds:
\begin{lem}\label{weakstearphi}
Suppose that the sequence $u_n(.)\rightharpoonup\mu(.,\alpha)$ in the nonlinear weak star sense, assume that $g$  is a continuous non decreasing function such that $g(u_n(.))\longrightarrow\theta$ strongly in $L^1(Q)$. Then, $\theta=g(\mu(.,\alpha))=g(u)$ where $u(t,x)=\displaystyle\int_0^1\mu(t,x,\alpha)d\alpha$.% 
\end{lem}
\begin{proof}
Let $v_n=g(u_n)$, since g is continuous, then the sequence $v_n$ is bounded in $L^\infty(Q)$,  so that $v_n(t,x)\stackrel{nl-*}{\rightharpoonup}\nu(t,x,\alpha)$ (where $\stackrel{nl-*}{\rightharpoonup}$ mean the convergence for weak star topology in  $L^\infty(Q)$) and $v_n\rightarrow \theta$ in $L^1(Q)$ and $\nu(t,x,\alpha):=g(\mu(t,x,\alpha))$ is an associated Young measure, since for all $h\in\mathcal{C}(\mathbb R,\mathbb R)$
\begin{align*}
h(v_n(t,x))=(h\circ g)(u_n)\rightharpoonup\displaystyle\int_0^1 h\circ g(\mu(.,\alpha))d\alpha=\displaystyle\int_0^1 h(\nu(t,x,\alpha)) d\alpha.
\end{align*}
Since $v_n$ tend to $\theta$ strongly, one deduces that $\nu(t,x,\alpha)=\theta(t,x)$ and $\nu$ does not depend on $\alpha$. Moreover, if $g$ is continuous and nondecreasing the level sets $g^{-1}(\{c\})$ are closed intervals of $\mathbb R$. Then for all $(t,x)\in Q$, 
\begin{align}
\mu(.,\alpha)\in g^{-1}(\{\theta(.)\})\Longrightarrow u(.)=\displaystyle\int\mu(.,\alpha)d\alpha\in g^{-1}(\{\theta(.)\})\Longrightarrow g(u(.))=\theta(.).
\end{align}
\end{proof}
\end{thm}
From now we give a	 ''discrete $L^2(0,T;H^1(\Omega))$'' compactness result (see e.g. \cite{EGH}).
\begin{lem}\label{convphi}
Consider a family of corresponding discrete functions $\wOdt$ satisfying the uniform bounds.
\begin{align}
\sum_{n=0}^{N}\delta t\sum_{K\in\mathcal{O}}m(K)(w^{N+1}_K)^2\leq C;\;\;\; \sum_{n=0}^{N}\delta t\sum_{K|L}\tau_{K|L}(\nabla_{\widehat{K|L}}w_{\mathcal{O}})^2\leq C,
\end{align}
where the discrete gradient $\nabla_{\widehat{K|L}}$ are defined by \eqref{gradiantdiscret}.\\Then there exists $w\in L^2(0,T;H^1(\Omega))$ such that, up to extraction of a subsequence, $\wOdt\rightarrow w$ in $L^2(Q)$ weakly and $\nabla_{\mathcal{O}}w\rightharpoonup\nabla w$ in $(L^2(Q))^\ell$ weakly.
\end{lem}
%\begin{proof}
%$\nabla_{\mathcal{O}}\phi(\uOdt)$ are bounded in $(L^2(Q))^\ell$ uniformly and $\phi(\uOdt)$ are uniformly bounded in $L^2(Q)$ Extracting weakly convergent subsequences, we get $\phi(\uOdt)\rightharpoonup\phi(u)$ weakly in $L^2$ and $\nabla_{\mathcal{O}}\phi(\uOdt)\rightharpoonup\chi$. Its remains to identify $\chi$ to the gradient of $\phi$ in the sens of distributions. Taking $\xi\in\mathcal{C}^\infty_0(Q)$ 
%\end{proof}
We wish to prove the convergence of the approximate solution $(\uOdt)$ to an entropy solution $u$ of $(P)$, i.e. we want to prove that there exists a limit $u$ and that  it satisfies  \eqref{ESzeroflux}. For that purpose, we prove first that $(\uOdt)$ tends in the nonlinear weak star sense to an entropy-process solution. 
\begin{prop}\label{conventropyprocess}(Convergence towards an entropy-process solution)
Under hypotheses \eqref{reg2}-\eqref{consistency}, let $\uOdt$ be the approximate solution of problem $(P_1)$ defined by  \eqref{esti0}, \eqref{esti1}. There exists an entropy-process solution $\mu$ of $(P_1)$ in the sense of Definition \ref{entrprosol} and a subsequence of $(\uOdt)_{\mathcal{O},\delta t}$ , such that:
\begin{enumerate}
\item The sequence $(\uOdt)_{\mathcal{O},\delta t}$ converges to $\mu$ in the nonlinear weak star sense.
\item Moreover $(\phi(\uOdt))_{\mathcal{O},\delta t}$ converges strongly in $L^2(Q)$ to $\phi(u)$  as $h,\delta t$ tend to zero and 
\item  $(\nabla_{\mathcal{O}}\phi(\uOdt))_{\mathcal{O},\delta t}\rightharpoonup\nabla \phi(u)$ in $(L^2(Q))^\ell$ weakly,
 \end{enumerate}
 where $u(t,x)=\int_0^1\mu(t,x,\alpha)d\alpha$. 
\end{prop} 
From this result, we deduce Theorem \ref{conv},  using additional regularity properties coming from \eqref{dimension}, \eqref{conv-uc} or \eqref{conv-umax} (see also Remark \ref{interessant} for variants of the concluding argument).
\begin{proof}[Proof of Proposition~\ref{conventropyprocess}]
Passage to the limit in the  continuous entropy inequality:\\
Recall that we have proved that $\upsilon_{\mathcal{O},n}(\xi)\rightarrow0$ when $(h,\delta t)\rightarrow0$ for $\xi\in\mathcal{C}^\infty([0,T[\times\mathbb R^\ell)$.  We follow step by step the passage to the limit for each term of the left hand side of \eqref{continousentropy}.\\
Because $\uOdt$ is bounded in $L^\infty(Q)$, by Theorem \ref{weakstar}, there exist $\mu\in L^\infty(Q\times(0,1))$ such that up to a subsequence, $(\uOdt)$ tends to $\mu$  in the nonlinear weak star sense as $\max(\delta t, h)\longrightarrow0$. We set $u(t,x)=\int_0^1\mu(t,x,\alpha)d\alpha$. Using the continuity of $\Phi_k(.)$ and $\eta_k(.)=|.-k|$, we prove that:
\begin{align*}
\displaystyle\int_0^T\!\int_\Omega\eta_k(\uOdt)\xi_tdxdt\rightarrow\displaystyle\int_0^1\int_0^T\!\int_\Omega\displaystyle\int_0^1\mu(t,x,\alpha)\xi_tdxdtd\alpha,
\end{align*}
\begin{align*}
\displaystyle\int_0^T\!\!\!\int_\Omega\displaystyle \Phi_k(\uOdt).\nabla\xi dxdt\rightarrow\int_0^T\!\!\!\int_\Omega\displaystyle\int_0^1\displaystyle \Phi_k(\mu)dxdtd\alpha.
\end{align*}
Due to \eqref{H1estima}  and by the Fr\'echet-Kolmogorov's theorem (due to the time and space translation on $\phi(\uOdt))$ we can apply lemma \ref{convphi} for $\wOdt=\phi(\uOdt)$.  Notice that in view of  Lemma \ref{weakstearphi}, it appears that $\phi(\mu)=\phi(u)$ where $u(t,x)=\int_0^1\mu(t,x,\alpha)d\alpha$. The Lipschtiz continuity of $\eta_{\phi}$ permits to have 
\begin{align*}
\displaystyle\int_0^T\!\!\!\int_\Omega\nabla_{\mathcal{O}}\eta_{\phi}(\phi(u_{\mathcal{O},n})).\nabla\xi\rightharpoonup\displaystyle\int_0^T\!\!\!\int_\Omega\nabla\eta_{\phi}(\phi(u)).\nabla\xi.
\end{align*}
We conclude that $\uOdt$ converge to an entropy-process solution $\mu$.
\end{proof}
\end{subsection}
\end{section}
%%%%%%%%%%%%%%%%%%%%%%%%%%%%%%%%%%%%%%%%%%%%%%%%%%%%%%%%%%%%%%%%%%%%%%%%%%%%%%%%%%%%%%%%%%%%%%%%%%%%%%%%%%%%%%%%%%%%%%%%%%%%%%%%%%%%%%%%%%%%%%%%%%%%%%%%%%%%%%%%%%%
%\section{Numerical results}
%In this part we present some numerical test to demonstrate the robustness of our schemes. We place ourselves in two-dimensional space $\Omega=[0,1]\times[0,1]$, for a simple one space dimension, on can refer to \cite{GAZ}. The rectangular mesh are used to discretize $(P)$.
%%%%%%%%%%%%%%%%%%%%%%%%%%%%%%%%%%%%%%%%%%%%%%%%%%%%%%%%%%%%%%%%%%%%%%%%%%%%%%%%%%%%%%%%%%%%%%%%%%%%%%%%%%%%%%%%%%%%%%%%%%%%%%%%%%%%%%%%%%%%%%%%%%%%%%%%%%%%%%%%%%
\section{Appendix 1}
We consider here a Banach  space $X$ (in application to the problem $(P)$, we will take $X=L^1(\Omega)$) and the multivalued operator $A:X\times X\longrightarrow X$ defined by its graph. We study the general evolution problem $u'+Au\ni h$, $u(0)=u_0$. In our application, $A$ is formally defined by $Au=\div f(u)-\Delta\phi(u)$ with zero-flux boundary condition.  In the sequel, we suppose that the operator $A$ is m-accretive and $u_0\in \overline{D}(A) $. We refer to \cite{BGP} for  definition and to \cite{BF,BG} and Appendix 2 for proof of these properties in our concrete setting which is our final purpose.  In relation with the classical notion of integral solution to the abstract evolution problem introduced in the thesis of Ph. B\'enilan , see e.g. \cite{BGP}, \cite{Mazones} we consider a new notion of solution called integral-process solution which depend on an additional variable $\alpha\in(0,1)$. The purpose here is to prove that the integral-process solution  of $(E)$ coincides with mild and integral solutions. Therefore the interest of  the notion of integral-process solution resides only in the fact that it may appear from some weak convergence arguments, see Appendix 2 for the example we have in mind. Let us recall the notion of mild solution. In the sequel, $||.||=||.||_X$ being the norm in $X$.
\begin{defn}
A mild solution of the abstract problem $u'+Au\ni h$ on $[0,T]$ is a function $u\in\mathcal{C}([0,T]; X)$ such that for $\sigma>0$ there is an $\sigma-$ discretization $D_A^N(t_0,....,t_N,h_1,...,h_N)$ of  $u'+Au\ni h$ on $[0,T]$ which has an $\sigma-$ approximate solution $v$ satisfying 
\begin{align}
||u(t)-v(t)||\leq \sigma \mbox{ for } t_0\leq t\leq t_N.
\end{align}
\end{defn}
%\begin{rem}
Recall that a $\sigma-$ approximate solution $v$ of $u'+Au\ni h$ on $[0,T]$ is the solution of an $\sigma-$ discretization $D_A^N(t_0,....,t_N,h_1,...,h_N)$:
\begin{align}
\frac{v_i-v_{i-1}}{t_i-t_{i-1}}+Av_i\ni h_i,\;\;\; i=1,2,...,N
\end{align}
where $h\approx\displaystyle\sum_{i=1}^N h_i1_{]t_{i-1},t_i]}$ and $|t_i-t_{i-1}|\leq\sigma$. Further,  $v$ is an $\sigma-$ approximate solution of the abstract initial value problem $(E)$ if also $t_0=0$ and $||v_0-u_0||\leq\sigma$.
%\end{rem}
\begin{thm}
Let  $A$ be m-accretive in $L^1(\Omega)$ and $u(0)\in \overline{D}(A)$. Then the abstract initial-value problem $u'+Au\ni h$ on $(0,T]$, $u(0)=u_0$ has a unique mild solution $u$ on $[0,T]$. Moreover $u$ is the unique function on $\mathcal{C}([0,T], X)$ such that for all $(\hat{u},z)\in A$
\begin{align}\label{intsolu}
||u(t)-\hat{u}||-||u(s)-\hat{u}||\leq\int_s^t\biggl[u(\tau)-\hat{u}, g(\tau)-z\biggr]d\tau 
\end{align}
for $0\leq s\leq t\leq T$.
\\
Here, $[a,b]:=\displaystyle\lim_{\lambda\downarrow 0} \frac{||a+\lambda b||-||a||}{\lambda}$ is the bracket on $X$  (see \cite{BGP}). In particular if $X=L^1$ then $[a,b]_{L^1(\Omega)}=\displaystyle\int_{\Omega}sign(a)bdx+\displaystyle\int_{\{a=0\}}|b|dx$.
\end{thm}
For the proof, we refer to \cite{BGP}.\\
A function $u$ satisfying \eqref{intsolu} is called integral solution. Here, we consider a seemingly more general notion of solution inspired by which is the object of this paper.%\\cite{AndreBK} \cite{VOV}, \cite{BG}.
\begin{defn}
Let  $A$ be an accretive operator and $g\in L^1(0,T;X)$. A function $v(t,\alpha)$ is an integral-process solution of abstract problem $v'+Av\ni g$ on $[0,T]$, $\nu(0,\alpha)=\nu_0$, if $v$ satisfy for all $(\hat{\nu},z)\in A$
\begin{align}\label{intprocsolu}
\displaystyle\int_0^1\!\!\biggl(||v(t,\alpha)\!-\hat{\nu}||\!-\!||v(s,\alpha)\!-\hat{\nu}||\biggr)d\alpha\!\leq\displaystyle\int_0^1\!\!\!\!\int_s^t\biggl[v(\tau,\alpha)-\hat{\nu}, g(\tau)-z\biggr]d\tau d\alpha
\end{align}
for $0< s\leq t\leq T$  and the initial condition is satisfied in the sense
\begin{align}\label{processunitial}
\mbox{ess-}\lim_{t\downarrow0}\int_0^1||v(t,\alpha)-\nu_0||d\alpha=0.
\end{align}
\end{defn}
Such generalization of the notion of integral solution is a purely technical hint, indeed, we show that  integral-process solutions coincide with the unique integral solution in the following sense.
\begin{thm}
Assume that $A$ be m-accretive in $X$ and $u_0\in\overline{D}(A)$, $u$ is an integral-process solution if and only if $u$ is independent on $\alpha$ and for all $\alpha$, $u(.,\alpha)$  coincide with the unique integral and mild solution.
\end{thm}
The result  will follow directly from the proposition given bellow.
\begin{prop}
Let  $A$ be an accretive operator. If $v$ is an integral-process solution of $v'+Av\ni g$ on $[0,T]$, $\nu(0,\alpha)\equiv\nu_0$ and $u$ is a mild solution of $u'+Au\ni h$  on $[0,T]$, $u(0)=u_0$  then
\begin{align}\label{intpro}
\displaystyle\int_0^1\!||u(t)-v(t,\alpha)||d\alpha\!\leq\displaystyle\int_0^1\!||u_0-\nu_0||d\alpha+\!\displaystyle\displaystyle\int_0^t\!\!\int_0^1\biggl[u(\tau)-v(\tau,\alpha), h(\tau)-g(\tau)\biggr]d\tau d\alpha
\end{align}
for a.e. $t\in[0,T]$,
\end{prop}
\begin{proof}
Let $u_k^n$, $k=1,..., N_n$ be a  solution of the $\sigma_n$ discretization $D_A(0=t_0^n,t_1^n,....,t_{N_n}^n)$ of $u'+Au\ni h$ on $[0,T]$. Set $\delta^n_k=t^n_k-t^n_{k-1}$ and let $0\leq a\leq b\leq T$. Since $v$ is an integral-process solution of $v'(t,\alpha)+Av\ni g$ we have:
\begin{align}\label{semi1}
&\displaystyle\int_0^1\biggl(||v(b,\alpha)-u_k^n||-||v(a,\alpha)-u_k^n||\biggr)d\alpha\nonumber\\&\leq\displaystyle\int_0^1\int_a^b\biggl[v(\tau,\alpha)-u^n_k, g(\tau)-h^n_k+\frac{u_k^n-u^n_{k-1}}{\delta^n_k}\biggr]d\alpha d\tau\nonumber\\&\leq\displaystyle\int_0^1\int_a^b\biggl[v(\tau,\alpha)-u^n_k, g(\tau)-h^n_k\biggr]d\alpha d\tau\nonumber\\&+\frac{1}{\delta^n_k}\displaystyle\int_0^1\int_a^b\biggl(||v(\tau,\alpha)-u_{k-1}^n||-||v(\tau,\alpha)-u_k^n||\biggr)d\alpha d\tau.
\end{align}
Where we have used the inequality
\begin{align}%\label{semi1}
\biggl[v(\tau,\alpha)-u^n_k, g(\tau)-h^n_k+\frac{u_k^n-u^n_{k-1}}{\delta^n_k}\biggr]&\leq\biggl[v(\tau,\alpha)-u^n_k, g(\tau)-h^n_k\biggr]\nonumber\\&+\frac{1}{\delta^n_k}\biggl(||v(\tau,\alpha)-u_{k-1}^n||-||v(\tau,\alpha)-u_k^n||\biggr)
\end{align}
which follows from the facts that $\biggl[X,Y+Z\biggr]\leq\biggl[X,Y\biggr]+\biggl[X,Z\biggr]$; $\biggl[X,eY\biggr]=e\biggl[X,Y\biggr]$ if $e>0$ and $\biggl[X,Y\biggr]\leq\displaystyle\frac{||X+eY||-||X||}{e}$. Multiplying \eqref{semi1} by $\delta_n^k$ and summing over $k=j+1,j+2,...,i$ we find that:
\begin{align}\label{semi2}
&\displaystyle\displaystyle\sum_{k=j+1}^i\displaystyle\int_0^1\delta^n_k \biggl(||v(b,\alpha)-u_k^n||-||v(a,\alpha)-u_k^n||\biggr)d\alpha\nonumber\\\leq&\displaystyle\displaystyle\sum_{k=j+1}^i\delta^n_k\displaystyle\int_0^1\int_a^b||\biggl[v(\tau,\alpha)-u^n_k, g(\tau)-h^n_k\biggr]d\tau d\alpha \nonumber\\&+\displaystyle\int_0^1\int_a^b\biggl(||v(\tau,\alpha)-u_j^n||-||v(\tau,\alpha)-u_i^n||\biggr) d\tau d\alpha.
\end{align}
Next, we assume that $\sigma_n\rightarrow0$ and the $\sigma_n-$ approximate solution of $u'+Au\ni h$ locally converge uniformly to the mild solution $u$ on $[0,T[$. Set 
\begin{align*}
\phi_n(\iota,\lambda,\alpha)=||v(\iota,\alpha)-u_{k_{(\lambda)}}^n||\mbox{ for } 0\leq \iota\leq T;\mbox{ where } k(\lambda) \mbox{ is defined by } t_{k_{(\lambda)}-1}^n<\lambda\leq t_{k_{(\lambda)}}^n.
\end{align*}
Then $\phi_n(\iota,\lambda,\alpha)\rightarrow||v(\iota,\alpha)-u(\lambda)||$ uniformly on $[0,T[\times[0,T[\times[0,1]$. Hence
\begin{align}
\biggl|||v(\iota,\alpha)-u(\lambda)||-||v(\iota,\alpha)-u_{k_{(\lambda)}}^n||\biggr|\leq||u_{k_{(\lambda)}}^n-u(\lambda)||\rightarrow0.
\end{align}
Therefore, if we choose $i,j$ depending on $n$ so that $t_j^n\rightarrow c$, $t_i^n\rightarrow d$ as $n\rightarrow\infty$ we have
\begin{align}
\displaystyle\sum\nolimits_{k=j+1}^i\delta_k^n||v(\iota,\alpha)-u_{k}^n||\rightarrow\int_c^d||v(\iota,\alpha)-u(\lambda)||d\lambda\mbox{ for } \iota\in[0,T].
\end{align}
Moreover with $\iota=\tau$, we get
\begin{align}\label{iota1}
\displaystyle\int_a^b||v(\tau,\alpha)-u^n_j||d\tau\rightarrow\displaystyle\int_a^b||v(\tau,\alpha)-u(c)||d\tau\mbox{ and }
\end{align}
\begin{align}\label{iota2}
\displaystyle\int_a^b||v(\tau,\alpha)-u^n_i||d\tau\rightarrow\displaystyle\int_a^b||v(\tau,\alpha)-u(d)||d\tau.
\end{align}
From now, let
\begin{align}
F_n(\lambda,\alpha)=\int_a^b\biggl[v(\tau,\alpha)-u_k^n,g(\tau)-h^n_k\biggr]d\tau\mbox{ for } t^n_{k-1}<\lambda\leq t_k^n.
\end{align}
Then
\begin{align}
&\biggl|\displaystyle\displaystyle\sum_{k=j+1}^i\biggl(\delta_k^nF_n(\lambda,\alpha)-\int_{t^n_{k-1}}^{t^n_k}\int_a^b\biggl[v(\tau,\alpha)-u_k^n,g(\tau)-h(\lambda)\biggr]d\tau d\lambda\biggr)\biggr|\nonumber\\
&\leq\displaystyle\sum_{k=j+1}^i\int_{t^n_{k-1}}^{t^n_k}\int_a^b||h_k^n-h(\lambda)||d\tau d\lambda\leq\sigma_n(b-a).
\end{align}
and therefore 
\begin{align}\label{semi6}
\mathop{\overline{\lim}}\limits_{n\rightarrow\infty}\!\!\displaystyle\sum_{k=j+1}^i\!\!\delta_k^n F_n(\lambda,\alpha)\!=\!\mathop{\overline{\lim}}\limits_{n\rightarrow\infty}\!\!\displaystyle\sum_{k=j+1}^i\!\int_{t^n_{k-1}}^{t_n^k}\!\int_a^b\biggl[v(\tau,\alpha)-u_k^n,g(\tau)-h(\lambda)\biggr]d\tau d\lambda.
\end{align}
Since $u_k^n\rightarrow u(\lambda)$ and $t^n_k\rightarrow \lambda$ as $n\rightarrow 0$ and  the bracket $\biggl[\cdot,\cdot\biggr]$ is the upper-semicontinuous, we deduce from \eqref{semi6} that
\begin{align}\label{semi7}
&\mathop{\overline{\lim}}\limits_{n\rightarrow\infty}\displaystyle\sum_{k=j+1}^i\delta_k^n\int_a^b\biggl[v(\tau,\alpha)-u_k^n,g(\tau)-h^n_k\biggr]d\tau=\mathop{\overline{\lim}}\limits_{n\rightarrow\infty}\int_{t_j^n}^{t_i^n}F_n(\lambda,\alpha) d\lambda \nonumber\\&\leq\int_c^d\int_a^b\biggl[v(\tau,\alpha)-u(\lambda),g(\tau)-h(\lambda)\biggr]d\tau d\lambda.
\end{align}
As previously, the convergence is uniform in $\alpha\in[0,1]$, therefore we can integrate in $\alpha$ under the limite in \eqref{iota1}, \eqref{iota2}, \eqref{semi7} and obtain
%finally using \eqref{intg1}, \eqref{intg2}, \eqref{semi6} we can pass to the limit in to \eqref{semi2} conclude that:
\begin{align}
&\displaystyle\int_0^1\int_c^d\biggl(||v(b,\alpha)-u(\lambda)||-||v(a,\alpha)-u(\lambda)||\biggr)d\lambda d\alpha\nonumber\\&\leq\displaystyle\int_0^1\int_c^d\int_a^b\biggl[v(\tau,\alpha)-u(\lambda),g(\tau)-h(\lambda)\biggr]d\tau d\lambda d\alpha\nonumber\\&+\displaystyle\int_0^1\int_c^d\biggl(||v(\tau,\alpha)-u(c)||-||v(\tau,\alpha)-u(d)||\biggr)d\tau d\alpha.
\end{align}
Now, we set:
\begin{align*}
&\varpi(s,t,\alpha)=\displaystyle\int_0^1||v(s,\alpha)-u(t)||d\alpha\\
&\Pi(s,t,\alpha)=\displaystyle\int_0^1\biggl[v(s,\alpha)-u(t),g(s)-h(t)\biggr]d\alpha.
\end{align*}
Recall that $u\in\mathcal{C}([0,T];X)$ and $\mbox{ess-}\lim_{t\downarrow0}\int_0^1||v(t,\alpha)-u_0||d\alpha=0$. Then, $v$ is  continuous a.e. for any Lebesgue point on $[0,T]$. The function $\varpi$ and $\Pi$ are continuous in $t$ and integrable in $s$
\begin{align}
\varphi(t,t)-\varphi(s,s)\leq\int_s^t\Pi(\tau)d\tau=\int_0^t\Pi(\tau)d\tau-\int_0^s\Pi(\tau)d\tau.
\end{align}
Then
\begin{align}
\Xi(t)=\varphi(t,t)-\int_0^t\Pi(\tau)d\tau\leq \varphi(s,s)-\int_0^s\Pi(\tau)d\tau=\Xi(s) \mbox{ for a.e. } t,s \in[0,T].
\end{align} 
The function $\Xi$ is continuous at $0^+$, therefore $\Xi(t)\leq\Xi(0)$. This is equivalent to \eqref{intpro}.
\end{proof}
%%%%%%%%%%%%%%%%%%%%%%%%%%%%%%%%%%%%%%%%%%%%%%%%%%%%%%%%%%%%%%%%%%%%%%%%%%%%%%%%%%%%%%%%%%%%%%%%%%%%%%%%%%%%%%%%%%%%%%%%%%%%%%%%%%%%%%%%%%%%%%%%%%%%%%%%%%%%%%%%%%%%%
\section{Appendix 2}
In this appendix, we apply the notion of integral-process solution to the problem $(P)$ and  present a way to prove uniqueness of entropy solution. In \cite{VOV}, the authors introduced a notion of entropy-process solution and using the doubling of variable method of Kruzhkov \cite{KRU} they  proved that entropy solution is the unique entropy-process solution. In our case, we were not able to use the same argument because we need that the entropy solution possess a strong boundary trace on the boundary in order that the doubling of variables apply (see \cite{BF}). Fortunately, under additional assumptions, we can ensure the desired boundary regularity for the associated stationary problem:  
\begin{equation*}
(S)\left \{\begin{array}{rll}
v+\div(f(v)-\nabla\phi(v))&=g  &\mbox{ in } \Omega,\\
\bigl(f(v)-\nabla\phi(v)\bigr).\eta&=0  &\mbox{ on } \partial\Omega.
\end{array} \right.
\end{equation*}
Therefore, firstly we compare the entropy-process solution $\mu$ of $(P)$ to the solution of $(S)$. This suggests the use of nonlinear semigroup theory; more precisely we find that $\mu$ is also an integral-process solution to $u'+Au=0$, $\mu(0,\alpha)=u_0$ with appropriately defined operator $A$. Then, proving the m-accretivity of $A$ and using the Appendix 1 we are able to conclude that $\mu$ is the unique mild and integral solution of the abstract evolution problem. At the last step, we use the result of \cite{BG} which says that such solution is the unique entropy solution of $(P)$.
\begin{prop}\label{carrilloprocessus}
 Let $\xi\in \mathcal{C}^\infty([0,T[\times\mathbb R^\ell)$, $\xi\geq0$. Then for all $k\in]u_{c},u_{\max}]$, for all $D\in\mathbb R^\ell$ and for all entropy-process solution $\mu$ of $(P)$, we have
\begin{align}\label{procarproc}
&\displaystyle\int_0^T\!\!\!\!\int_\Omega\int_0^1\left\{|\mu-k|\xi_t+sign(\mu-k)\Bigl[f(\mu)-f(k)].\nabla\xi\right\} dxdtd\alpha\nonumber\\&-\displaystyle\int_0^T\!\!\!\!\int_\Omega sign(u-k)\biggl(\nabla\phi(u)-D\biggr).\nabla\xi dxdt\nonumber\\&+\displaystyle\int_\Omega|u_0-k|\xi(0,x)dx+\displaystyle\int_0^T\!\!\!\!\int_{\partial\Omega} \left|(f(k)-D).\eta(x)\right|\xi d{\mathcal{H}}^{\ell-1}(x)dt\nonumber\\&\geq\mathop{\overline{\lim}}\limits_{\sigma\rightarrow0}\frac{1}{\sigma}\displaystyle\int\!\!\int_{Q\cap\left\{-\sigma<\phi(u)-\phi(k)<\sigma\right\}}\nabla\phi(u).\biggl(\nabla\phi(u)-D\biggr)\xi dxdt.
\end{align}
\end{prop}
\begin{proof}
The proof follows the arguments of \cite{BG}. Let us recall that if $\mu(t,x,\alpha)$ is entropy-process solution and \eqref{f} hold, then $u(t,x)=\displaystyle\int_0^1\mu(t,x,\alpha)d\alpha$ satisfies in the weak sense for all  $k\in[u_c, u_{\max}]$ and all $D\in\mathbb R^\ell$:
\begin{equation*}
(P_k)\left \{\begin{array}{rll}
(u-k)_t+\div \biggl[\biggl(\displaystyle\int_0^1f(\mu)d\alpha-\nabla\phi(u)\biggr)-\biggl(f(k)-D\biggr)\biggr]&=0  &\mbox{ in } Q,\\
 \mbox{ess-}\displaystyle\lim_{t\rightarrow0^+}(u(t,x)-k)&= u_{0}(x)-k  &\mbox{ on }\Omega,\\
\biggl[\biggl(\displaystyle\int_0^1f(\mu)d\alpha-\nabla\phi(u)\biggr)-\biggl(f(k)-D\biggr)\biggr].\eta&=-(f(k)-D).\eta  &\mbox{ on } \Sigma.
\end{array} \right.
\end{equation*}
Take the test function $sign_{\sigma}(\phi(u)-\phi(k))\xi=H_\sigma(\phi(u)-\phi(k))\xi$ in the weak formulation of this problem with $\xi\in \mathcal{C}^\infty([0,T)\times\mathbb R^\ell)$. Using the formalism of \cite{ALT}, we have
\begin{align}\label{pro1}
&\displaystyle\int_0^T\left\langle (u-k)_t,H_{\sigma}(\phi(u)-\phi(k))\xi\right\rangle_{H^1(\Omega)^*,H^1(\Omega)} dt\nonumber\\&-\int_0^T\!\!\!\!\int_\Omega H_{\sigma}(\phi(u)-\phi(k))\Bigl[\Bigl(\int_0^1f(\mu)d\alpha-\nabla\phi(u)\Bigr)-\Bigl(f(k)-D\Bigr)\Bigr].\nabla\xi\nonumber\\&-\int_0^T\!\!\!\!\int_\Omega\xi\Bigl[\Bigl(\int_0^1f(\mu)d\alpha-\nabla\phi(u)\Bigr)-\Bigl(f(k)-D\Bigr)\Bigr].\nabla H_{\sigma}(\phi(u)-\phi(k))\nonumber\\&-\int_0^T\!\!\!\!\int_{\partial\Omega}H_{\sigma}(\phi(u)-\phi(k))(f(k)-D).\eta\xi=0.
\end{align}
 By the weak chain rule (see \cite{ALT})
\begin{align}
\displaystyle\int_0^T\left\langle (u-k)_t,H_{\sigma}(\phi(u)-\phi(k))\xi\right\rangle_{H^1(\Omega)^*,H^1(\Omega)}dt=&-\displaystyle\int_0^T\!\!\displaystyle\int_\Omega I_\sigma(u)\xi_t dtdx\nonumber\\&-\displaystyle\int_\Omega I_\sigma(u_0)\xi(0,x)dx
\end{align}
where: $H_{\sigma}(r)=\left \{\begin{array}{rll}
  1  &\mbox{ if } \; r>\sigma,\\
  \frac{r}{\sigma} &\mbox{ if } \; |r|\leq\sigma,\\
  -1  &\mbox{ if } \; r<-\sigma,
\end{array} \right.$ and
\begin{align}
I_\sigma:z\longmapsto\displaystyle\int_k^zH_{\sigma}(\phi(s)-\phi(k))ds\longrightarrow|z-k| \mbox{ as }\sigma\rightarrow0.
\end{align}
Then, after passing to the limit as $\sigma\rightarrow0$, we have
\begin{align}
\int_Q|u-k|\xi_tdxdt&=-\int_Qsign(u-k)(u-k)\xi_tdxdt\nonumber\\&=-\int_Q sign(u-k)\biggl(\int_0^1\mu d\alpha-k\biggr)\xi_tdxdt
\end{align}
Now,notice that because $k\in]u_c,u_{\max}[$ and because $\phi(\mu(\alpha))=const$ on $[0,1]$ we find that $sign(\mu(\alpha)-k)$ is constant on $[0,1]$ equal to $sign(u-k)$
Then , we see that
\begin{align}
\int_Q|u-k|\xi_tdxdt=-\int_Q\int_0^1|\mu-k|\xi_td\alpha dxdt
\end{align}
Similarly, we see 
\begin{align}
&\displaystyle\int_0^T\displaystyle\int_0^1\int_\Omega sign(u-k)\Bigl[f(\mu)-f(k)\Bigr].\nabla\xi dxdt\nonumber\\&=\displaystyle\int_0^T\int_\Omega\int_0^1 sign(\mu-k)\Bigl[f(\mu)-f(k)\Bigr].\nabla\xi dxdtd\alpha
\end{align}
For treatment of the others terms, we refer to \cite{BG}.
\end{proof}
Let us firstly prove that the initial datum is satisfied in the sense of \eqref{processunitial} (see Appendix 1). This means that the entropy-process solution satisfies the initial condition of integral-process solution. 
\begin{lem}
Let $v$ be an entropy-process solution of $(P)$ with initial datum $v_0\in L^\infty$. Then the initial datum is taken in the following sense:
\begin{align}
\displaystyle\lim_{s\downarrow0}\displaystyle\int_0^s\int_{\Omega}\int_0^1|v-v_0| dt dxd\alpha=0.
\end{align} 
\end{lem}
\begin{proof}
The proof follows the one of Panov in (\cite[Proposition 1]{PAN2}). For $c\in\mathbb R$ and $s>0$, consider the functions
\begin{align}\label{ini}
\Psi_s(.,c):x\in\Omega\longmapsto\frac{1}{s}\int_0^s\int_0^1|v(t,x,\alpha)-c|dtd\alpha.
\end{align}
Because $v$ is bounded, the set $(\Psi_s(.,c))_{s>0}$ is bounded in $L^\infty(\Omega)$. Therefore for any sequence $s_m\rightarrow0$, there exists a subsequence such that for all $c\in\mathbb Q$, $(\Psi_s(.,c))_{s>0}$ converges in $L^\infty(\Omega)$ weak star to some limit denoted by $(\Psi(.,c))$. Fix $\xi\in\mathcal{D}(\Omega)^+$. From Remark \ref{entrprosol} with test function $\tilde{\xi}(t,x):=(1-\frac{t}{s})^+\xi(x)$
we readily infer the inequalities
\begin{align}\label{in}
\forall c\in\mathbb Q, \int_\Omega\Psi(x,c)\xi(x)dx\leq\int_\Omega|u_0-c|\xi(x)dx.
\end{align} 
By the density argument, we extend \eqref{in} to all $\xi\in L^1(\Omega)$, $\xi\geq 0$. Now for all $\epsilon>0$, there exists a number $N(\epsilon)\in\mathbb N$, a collection $(c_i^\epsilon)_{j=1}^{N(\epsilon)}\subset\mathbb Q$ and a partition of $\Omega$ into disjoint union of measurable sets $\Omega_1^\epsilon,...,\Omega_{N(\epsilon)}^\epsilon$ such that $||v_0-v_0^\epsilon||_{L^1}\leq \epsilon$, where $v_0^\epsilon:=\displaystyle\sum_{j=1}^{N(\epsilon)}c_j^\epsilon 1_{\Omega_j^\epsilon}$. Because $1_\Omega=\displaystyle\sum_{j=1}^{N(\epsilon)}1_{\Omega_j^\epsilon}$, applying \eqref{in} with $c=c_j^\epsilon$ and $\xi=1_{\Omega_j^\epsilon}$ we deduce
\begin{align*}
\displaystyle\lim_{m\rightarrow\infty}\frac{1}{s_m}\int_0^{s_m}\int_\Omega\int_0^1|v-v_0^\epsilon|dtdxd\alpha&=\displaystyle\lim_{m\rightarrow\infty}\int_\Omega\displaystyle\sum_{j=1}^{N(\epsilon)}\Psi_{s_m}(x,c_j^\epsilon)1_{\Omega_j^\epsilon}dx\\&=\int_\Omega\displaystyle\sum_{j=1}^{N(\epsilon)}\Psi(x,c_j^\epsilon)1_{\Omega_j^\epsilon}dx\\&\leq\int_\Omega\displaystyle\sum_{j=1}^{N(\epsilon)}|v_0-c_j^\epsilon|1_{\Omega_j^\epsilon}dx=||v_0-v_0^\epsilon||_{L^1}\leq\epsilon.
\end{align*} 
Using once more the bound $||v_0-v_0^\epsilon||_{L^1}\leq\epsilon$ (in the first term of the previous calculation), we can send $\epsilon$ to zero and infer the analogue of \eqref{ini}, with a limit taken along some subsequence of $(s_m)_{m>1}$. Because $(s_m)_{m>1}$ was an arbitrary sequence convergent to zero, \eqref{ini} is justified.
\end{proof}
Now it remains  to prove that the entropy-process solution an is integral-process solution. Let us define the (possibly multivalued) operator $A_{f,\phi}$ by it resolvent
\begin{equation*}
(v,z)\!\in A_{f,\phi}=\!\left \{\begin{array}{ll}
 \!\! v \mbox{ such that } v \mbox{ is an entropy solution of } (S_1), \mbox{ with } g=v+z.\\
\mbox{ and strong $L^1$ trace of } (f(u)-\nabla\phi(u)).\eta|_{\partial\Omega}\mbox{ exists}\\
\mbox{ and equal to zero}.
\end{array} \right\}.
\end{equation*}
\begin{defn}
The normal component of the flux $\mathcal{F}[u]=(f(u)-\nabla\phi(u)).\eta$ has a $L^1$ strong trace $\gamma \mathcal{F}[\hat{u}]\in L^1_{Loc}(\partial\Omega)$, at boundary $\partial\Omega$ if
\begin{align}
\lim_{s\rightarrow0}\frac{1}{s}\int_0^s\int_{\hat{x}\in\partial\Omega}\xi(\hat{x})|\mathcal{F}[u](s,\hat{x})-\gamma \mathcal{F}[u](\hat{x})|d\hat{x}d\tau=0.
\end{align} 
\end{defn}
After having defined this operator, we present the following results.
\begin{thm}\label{unicite}
Assume that $A_{f,\phi}$ is m-acccretive densely defined on $L^1(\Omega;[0,u_{\max}])$. Then the entropy-process solution is the unique entropy solution.
\end{thm}
Before turning to the proof of Theorem \ref{unicite}, lets us present three cases where it applies. %(this corresponds to the different cases in Theorem \ref{unite} .
\begin{prop}\label{uc-umax}
Assume \eqref{dimension} holds. Then, $A_{f,\phi}$ is m-acccretive densely defined on $L^1(\Omega;[0,u_{\max}])$.
\end{prop}
For the proof, we refer to \cite[Proposition 4.10]{BG}.
\begin{prop}\label{prouc}
Assume that, \eqref{conv-uc}, holds.  Then $A_{f,\phi}$ is m-acccretive densely defined on $L^1(\Omega;[0,u_{\max}])$. 
\end{prop}
\begin{proof}(sketch)
The proof is essentially the same as in \cite{BF}, where the case $\phi=Id$ has been investigated. For general $\phi$ satisfying $f\circ\phi^{-1}\in\mathcal{C}^{0,\alpha}, \alpha>0$ we adapt the result of Lieberman \cite{Liberman}. As $\phi$ is bijective, we set $w=\phi(u)$ and rewrite the stationary problem as:
\begin{align*}
\div(f\circ\phi^{-1}(w)-\nabla w)=g(x)-\phi^{-1}(w)\Rightarrow \div(B(w,\nabla w))=F(x,w),
\end{align*}
where $B$ and $F$ satisfies the hypothesis of \cite{Liberman}, then $w=\phi(u)\in\mathcal{C}^{0,\alpha}(\bar{\Omega}), \alpha>0$ and $u\in\mathcal{C}^{0,\alpha}(\bar{\Omega})$.We deduce that $(f(u)-\nabla(\phi(u)) \in\mathcal{C}(\bar{\Omega}).$  
\end{proof}
\begin{prop}\label{proumax}
Assume that \eqref{conv-umax} holds. Then $A_{f,\phi}$ is m-acccretive densely defined on $L^1(\Omega;[0,u_{\max}])$.
\end{prop}
For the proof, we refer to \cite{BFK1,VAS,PAN1} where the existence of strong trace of $f(u)$ has been proved for pure conservation laws.\\
In the sequel, we concentrate on the proof of Theorem \ref{unicite} in the case \eqref{dimension} holds. The other cases are similar, using the hint of \cite{BF}\footnote{Notice that the $\mathcal{C}^2$ regularity assumption on $\partial\Omega$, made in \cite{BF} is easily bypassed using the variant presented in \cite[Section 4]{BI}}.
\begin{proof}[Proof of Theorem ~\ref{unicite}]
Now, we apply the doubling of variables \cite{KRU} in the way of \cite{BF, BG}. We consider $\mu=\mu(t,x,\alpha)$ an entropy-process solution of $(P)$ and $v=v(y)$  an entropy solution of $(S)$ using in the definition of $A_{f,\phi}$. Consider nonnegative function $\xi=\xi(t,x,y)$ having the property that $\xi(.,.,y)\in \mathcal{C}^\infty([0,T)\times\overline{\Omega})$ for each $y\in\overline{\Omega}$, $\xi(t,x,.)\in \mathcal{C}_0^\infty(\overline{\Omega})$ for each $(t,x)\in[0,T)\times\overline{\Omega}$. Let us denote the sets on which the diffusion term for the first, respectively for the second solutions degenerate by
 $$\Omega_x=\displaystyle\left\{x\in\Omega ; \mu(t,x,\alpha)\in [0,u_c]\displaystyle\right\}; \Omega_y=\displaystyle\left\{y\in\Omega; v(y)\in [0,u_c]\displaystyle\right\}.$$ We denote by  $\Omega_x^c$ respectively $\Omega_y^c$ their complementaries in $\Omega$.
In \eqref{procarproc}, take $\xi=\xi(t,x,y)$, $k=u(y)$, $D=\phi(u)_y$ and  integrate over $\Omega_y^c\times[0,1]$. We get
\begin{align}\label{u10}
&\int_{\Omega_y^c}\int_0^T\!\!\!\int_{x\in\Omega}\displaystyle\int_0^1\displaystyle\left\{|\mu-v|\xi_t+sign(\mu-v)\Bigl[f(\mu)-f(v)].\xi_x\displaystyle\right\}d\alpha dxdtdy\nonumber\\
-&\int_{\Omega_y^c}\int_0^T\!\!\!\int_{x\in\Omega}sign(u-v)\biggl(\phi(u)_x-\phi(v)_y\biggr).\xi_xdxdtdy\nonumber\\&+\displaystyle\int_{\Omega_y^c}\int_0^T\!\!\!\int_{x\in\partial\Omega}\left|(f(v)-\phi(v)_y).\eta(x)\right|\xi d\sigma dtdy+\displaystyle\int_{\Omega_y^c}\int_{x\in\Omega}|u_0-v|\xi(0,x,y)dxdy\nonumber\\&
\geq\mathop{\overline{\lim}}\limits_{\sigma\rightarrow0}\frac{1}{\sigma}\displaystyle\int_{\Omega_y^c}\int_0^T\!\!\!\int_{x\in\Omega\cap\left\{-\sigma<\phi(u)-\phi(v)<\sigma\right\}}\phi(u)_x(\phi(u)_x-\phi(v)_y)\xi dxdtdy.
\end{align}
In the same way, in \eqref{entrprosol} take $\xi=\xi(t,x,y)$, $k=v(y)$,  integrate over $\Omega_y$, and use the fact that $\phi(v)_y=0$ in $\Omega_y$. We get
\begin{align}\label{u20}
&\int_{\Omega_y}\int_0^T\!\!\!\int_{x\in\Omega}\displaystyle\int_0^1\left\{|\mu-v|\xi_t+sign(\mu-v)\Bigl[f(\mu)-f(v)\Bigr].\xi_x\right\}d\alpha dxdtdy\nonumber\\
-&\int_{\Omega_y}\int_0^T\!\!\!\int_{x\in\Omega}sign(u-v)\biggl(\phi(u)_x-\phi(v)_y\biggr).\xi_xdxdtdy\nonumber\\&+
\displaystyle\int_{\Omega_y}\int_0^T\!\!\!\int_{x\in\partial\Omega}\left|(f(v)-\phi(v)_y).\eta(x)\right|\xi d\sigma dtdy\nonumber\\&+\displaystyle\int_{\Omega_y}\!\int_{x\in\Omega}|u_0-v|\xi(0,x,y)dxdy\geq0.
\end{align}
Since $\Omega=\Omega_x\cup{\Omega_x^c}$, by adding \eqref{u10} to \eqref{u20}  we obtain:
\begin{align}\label{u30}
&\int_{y\in\Omega}\int_0^T\!\!\!\int_{x\in\Omega}\displaystyle\int_0^1\displaystyle\left\{|\mu-v|\xi_t+sign(\mu-v)\Bigl[f(\mu)-f(v)].\xi_x\displaystyle\right\}d\alpha dxdtdy\nonumber\\
-&\int_{y\in\Omega}\int_0^T\!\!\!\int_{x\in\Omega}sign(u-v)\biggl(\phi(u)_x-\phi(v)_y\biggr).\xi_xdxdtdy\nonumber\\&+\displaystyle\int_{\Omega}\int_0^T\!\!\!\int_{x\in\partial\Omega}\left|(f(v)-\phi(v)_y).\eta(x)\right|\xi d\sigma dtdy+\displaystyle\int_{\Omega}\int_{x\in\Omega}|u_0-v|\xi(0,x,y)dxdy\nonumber\\&
\geq\mathop{\overline{\lim}}\limits_{\sigma\rightarrow0}\frac{1}{\sigma}\displaystyle\int_{\Omega_y^c}\int_0^T\!\!\!\int_{x\in\Omega\cap\left\{-\sigma<\phi(u)-\phi(v)<\sigma\right\}}\phi(u)_x(\phi(u)_x-\phi(v)_y)\xi dxdtdy.
\end{align}
In the entropy formulation of $(S)$, take $\xi=\xi(t,x,y)$, $k=\mu(t,x,\alpha)$, $D=\phi(\mu)_x$ and integrate over $(t,x,\alpha)\in(0,T)\times{\Omega_x^c}\times(0,1)$
\begin{align}\label{u40}
&\int_{\Omega_x^c}\int_0^T\!\!\!\displaystyle\int_0^1\displaystyle\int_{y\in\Omega} sign(v-\mu)\Bigl[f(v)-f(\mu)].\xi_yd\alpha dxdtdy\nonumber\\
-&\int_{\Omega_x^c}\int_0^T\!\!\!\displaystyle\int_0^1\int_{y\in\Omega}sign(v-\mu)\biggl(\phi(v)_y-\phi(u)_x\biggr).\xi_yd\alpha dxdtdy\nonumber\\+&
\int_0^T\int_{\Omega_x^c} \displaystyle\int_0^1\int_{y\in\Omega} sign(v-\mu)(v-g(y))\xi d\alpha dxdtdy\nonumber\\&+\displaystyle\int_0^T\int_{\Omega_x^c}\displaystyle\int_0^1\int_{y\in\partial\Omega}\left|(f(\mu)-\phi(\mu)_x).\eta(y)\right|\xi d\alpha d\sigma dxdt\nonumber\\&\geq \mathop{\overline{\lim}}\limits_{\sigma\rightarrow0}\frac{1}{\sigma}\displaystyle\int_{\Omega_x^c}\int_0^T\int_{y\in\cap\left\{-\sigma<\phi(u)-\phi(v)<\sigma\right\}}\displaystyle\int_0^1\phi(v)_y(\phi(v)_y-\phi(\mu)_x)\xi d\alpha dydtdx.
\end{align}
Since $v(y)$ is entropy solution, then  take in the entropy dissipative formulation of (S) $\xi=\xi(t,x,y)$, $k=\mu(t,x,\alpha)\in]u_c,u_{\max}[$, integrate over $(0,T)\times\Omega_x\times(0,1)$ and use the fact that $\phi(\mu)_x=\phi(u)_x=0$ in $(0,T)\times\Omega_x$.
\begin{align}\label{u50}
&\int_{\Omega_x}\int_0^T\!\!\!\displaystyle\int_0^1\displaystyle\int_{y\in\Omega} sign(v-\mu)\Bigl[f(v)-f(\mu)].\xi_yd\alpha dxdtdy\nonumber\\
-&\int_{\Omega_x}\int_0^T\!\!\!\displaystyle\int_0^1\int_{y\in\Omega}sign(v-\mu)\biggl(\phi(v)_y-\phi(u)_x\biggr).\xi_yd\alpha dxdtdy\nonumber\\+&
\int_0^T\int_{\Omega_x} \displaystyle\int_0^1\int_{y\in\Omega} sign(v-\mu)(v-g(y))\xi d\alpha dxdtdy\nonumber\\&+\displaystyle\int_0^T\int_{\Omega_x}\displaystyle\int_0^1\int_{y\in\partial\Omega}\left|(f(\mu)-\phi(\mu)_x).\eta(y)\right|\xi d\alpha d\sigma dxdt\geq0
\end{align}
By adding \eqref{u40} to \eqref{u50}, we obtain
\begin{align}\label{u60}
&\int_{\Omega}\int_0^T\!\!\!\displaystyle\int_0^1\displaystyle\int_{y\in\Omega} sign(v-\mu)\Bigl[f(v)-f(\mu)].\xi_yd\alpha dxdtdy\nonumber\\
-&\int_{\Omega}\int_0^T\!\!\!\displaystyle\int_0^1\int_{y\in\Omega}sign(v-\mu)\biggl(\phi(v)_y-\phi(u)_x\biggr).\xi_yd\alpha dxdtdy\nonumber\\+&
\int_0^T\int_{\Omega} \displaystyle\int_0^1\int_{y\in\Omega} sign(v-\mu)(v-g(y))\xi d\alpha dxdtdy\nonumber\\&+\displaystyle\int_0^T\int_{\Omega}\displaystyle\int_0^1\int_{y\in\partial\Omega}\left|(f(\mu)-\phi(\mu)_x).\eta(y)\right|\xi d\alpha d\sigma dxdt\nonumber\\&\geq \mathop{\overline{\lim}}\limits_{\sigma\rightarrow0}\frac{1}{\sigma}\displaystyle\int_{\Omega_x^c}\int_0^T\int_{y\in\cap\left\{-\sigma<\phi(u)-\phi(v)<\sigma\right\}}\displaystyle\int_0^1\phi(v)_y(\phi(v)_y-\phi(\mu)_x)\xi d\alpha dydtdx.
\end{align}
Now, sum  \eqref{u30} and \eqref{u60} to obtain
\begin{align}\label{u70}
&\displaystyle\int_0^1\int_0^T\!\!\int_\Omega\!\int_\Omega |\mu-v|\xi_td\alpha dydxdt+\displaystyle\int_\Omega\int_\Omega|u_0-v|\xi(0,x,y)dxdy\nonumber\\&+\int_0^1\int_0^T\!\!\!\int_\Omega\!\int_\Omega sign(v-u)\Bigl[f(\mu)-f(v)\Bigr].(\xi_x+\xi_y)d\alpha dydxdt\nonumber\\
-&\int_0^T\!\!\!\displaystyle\int_0^1\int_{\Omega}\int_{\Omega}sign(u-v)\biggl(\phi(u)_x-\phi(v)_y\biggr).(\xi_x+\xi_y)d\alpha dydxdt\nonumber\\&+\displaystyle\int_0^T\int_{x\in\partial\Omega}\!\int_\Omega \left|(f(v)-\phi(v)_y).\eta(x)\right|\xi d\sigma dtdy\nonumber\\&+\int_0^1\int_0^T\!\!\!\int_\Omega\!\int_{y\in\partial\Omega}\left|(f(\mu)-\phi(\mu)_x).\eta(y)\right|\xi dyd\sigma dt\nonumber\\&+\int_0^1\int_0^T\!\!\!\int_\Omega\!\int_\Omega sign(v-\mu)(v-g(y))\xi dydxdtd\alpha \nonumber\\&\geq\mathop{\overline{\lim}}\limits_{\sigma\rightarrow0}\frac{1}{\sigma}\displaystyle\int_0^T\int\!\!\!\int_{{{\Omega_x^c}\times{\Omega_y^c}}\cap\left\{-\sigma<\phi(v)-\phi(u)<\sigma\right\}}|\phi(v)_x-\phi(u)_y|^2\xi dydxdt\geq0.
\end{align}
Next, following the idea of \cite{BF} in the simple one-dimensional setting, we consider the test function $\xi(t,x,y)\!=\!\theta(t)\rho_n(x,y)$, where $\theta\!\in\!\mathcal{C}_0^\infty(0,T)$, $\theta\!\geq\!0$, $\rho_n(x,y)\!=\!\delta_n(\Delta)$ and  $\Delta\!=\!(1-\frac{1}{n(b-a)})x-y+\frac{a+b}{2n(b-a)}$. Then, $\rho_n\in \mathcal{D}(\overline\Omega\times\overline\Omega)$ and $\rho_{{n}_{|_{\Omega\times\partial\Omega}}}(x,y)=0$. Due to this choice $$\displaystyle\int_0^T\!\!\!\int_{x\in\Omega}\int_{y\in\partial\Omega}\int_0^1\left|(f(\mu)-\phi(u)_x).\eta(y)\right|\rho_n\theta dyd\sigma dt=0.$$
By the Proposition  \ref{uc-umax} and the  definition of $A_{f,\phi}$, we prove that for the stationary problem, $(f(v)-\phi(v)_y)\in \mathcal{C}_0([a,b])$. Therefore we have\\ $\left|(f(v)-\phi(v)_y).\eta(x)\right|\longrightarrow 0$ when $x\rightarrow y$, i.e, as $n\longrightarrow\infty$. We conclude that $$\lim_{n\rightarrow\infty}\displaystyle\int_0^T\!\!\!\int_{x\in\partial\Omega}\int_{y\in\Omega}\left|(f(v)-\phi(v)_y).\eta(x)\right|\rho_n\theta dyd\sigma dt=0.$$
It remains to study the limit, as $n\rightarrow\infty$
\begin{equation*}
I_n=\displaystyle\int_0^1\int_0^T\!\!\int_\Omega\!\int_\Omega \theta sign(\mu-v)\Bigl[(f(\mu)-\phi(u)_x)-(f(v)-\phi(v)_y)\Bigr].\bigl((\rho_n)_x+(\rho_n)_y\bigr) dydxdt.
\end{equation*}
We use the change of variable $(x,y)\mapsto (x,z)$ with $z=n(x-y)-\frac{1}{b-a}x+\frac{a+b}{b-a},$
%\begin{eqnarray*}
 \begin{multline}
 I_n%&
 =\frac{2}{b-a}\displaystyle\!\int_0^1\int_{-1}^1\!\int_0^T\!\!\!\int_\Omega \!sign(\mu-v)\Bigl[(f(\mu)-\phi(u)_x)-(f(v)-\phi(v)_y)\Bigr].\delta_n^{\prime}(z)\theta dxdtdzd\alpha\\
   %&
   =\frac{2}{b-a}\displaystyle\int_0^1\int_{-1}^1\int_0^T\!\!\!\int_a^b sign(\mu(t,x,\alpha)-v_n(x,z)) \hspace{4pt}  \\
      \Bigl[p(t,x,\alpha)-q_n(x,z,\alpha)\Bigr]\delta_n^{\prime}(z)\theta(t)dxdtdzd\alpha,
\end{multline}
where $v_n(x,z):=v(y)$, $p(t,x,\alpha):=f(\mu)-\phi(u)_x$ and $q_n:=f(v)-\phi(v)_y$.
For $z$ given, $v_n(.,z)$ converges to $v(.)$ in $L^1$ and $q_n(.,z)$ converges to $q(.):=f(v)-\phi(v)_x$ in $L^1$.  We deduce that for all $z\in[-1,1]$ (see  \cite{BF} also \cite{BG})
%We deduce that for
\begin{align*}
 K_n(z,\alpha):=\displaystyle\int_0^1\!\!\!\displaystyle\int_{Q}sign(w_n(t,x,z))h_n(t,x,\alpha)dxdtd\alpha\displaystyle\longrightarrow_{n\rightarrow\infty}&\displaystyle\int_0^1\!\!\!\displaystyle\int_Q sign(w)hdxdt\\&=:K=\mbox{const},
\end{align*}
where $w_n:=\mu-v_n$, $h_n:=p-q_n$ and $h:=p-q$.
Then $K_n(.)$ converges to $K$ independently on $z$. Moreover, from the definition of $K_n$ one finds easily the uniform $L^\infty$ bound $|K_n|\leq 2(||p||_{L^1(Q)}+T||q||_{L^1(\Omega)})$, for $n$ large enough. Hence by the Lebesgue theorem,
\begin{equation*}
\lim_{n\rightarrow\infty}\int_{-1}^1 K_n(z)\delta'(z)=K\int_{-1}^1\delta'(z)=0.
\end{equation*}
We have shown that the limit of $I_n$ equals zero. The passage to the limit in other terms in \eqref{u70} is straightforward. Finally \eqref{u70} gives for $n\longrightarrow\infty$
\begin{equation*}
\int_0^1\int_0^T\!\!\!\int_\Omega|\mu(t,x,\alpha)-v(y)|\theta'(t)dxdtd\alpha+\int_0^1\int_0^T\!\!\!\int_\Omega sign(v-\mu)(v-g)\theta dxdtd\alpha\geq 0.
\end{equation*}
Hence
\begin{equation*}
\frac{d}{dt}\int_0^1||\mu(t,\alpha)-v||_{L^1(\Omega)}d\alpha\leq\displaystyle\int_0^1\int_\Omega sign(\mu-v)(v-g)dx\mbox{ in } \mathcal{D}'(0,T) .
\end{equation*}
Thus, $v$ is an  integral-process solution of $(E)$ with $A=A_{f,\phi}$. 
\end{proof}
%Now, the claim of Theorem \ref{unicite} is a direct consequence of the fact that the integral-process solution is a unique integral solution (see Apendix 1) and then is an entropy solution of $(P)$.
%%%%%%%%%%%%%%%%%%%%%%%%%%%%%%%%%%%%%%%%%%%%%%%%%%%%%%%%%%%%%%%%%%%%%%%%%%%%%%%%%%%%%%%%%%%%%%%%%%%%%%%%%%%%%%%%%%%%%%%%%%%%%%%%%%%%%%%%%%%%%%%%%%%%%%%%%%%%%%%%%
%\section{Conclusion}
\subsection*{Acknowledgment}
I  would like to thank Boris Andreianov and Petra Wittbold for their through readings and helpful remarks which helped me improve this paper.
%Many thanks to our \TeX-pert for developing this class file.
%%%%%%%%%%%%%%%%%%%%%%%%%%%%%%%%%%%%%%%%%%%%%%%%%%%%%%%%%%%%%%%%%%%%%%%%%%%%%%%%%%%%%%%%%%%%%%%%%%%%%%%%%%%%%%%%%%%%%%%%%%%%%%%%%%%%%%%%%%%%%%%%%%%%%%%%%%%%%%%%%

% ------------------------------------------------------------------------
\end{document}